\documentclass[12pt]{amsart}
\usepackage{amsxtra}
\usepackage{amssymb}
\newcommand{\cleqn}{\setcounter{equation}{0}}
\newcommand{\clth}{\setcounter{theorem}{0}}
\newcommand {\sectionnew}[1]{\section{#1}\cleqn\clth}

\theoremstyle{plain}
\newtheorem{theorem}{Theorem}[section]
\newtheorem{lemma}[theorem]{Lemma}
\newtheorem{proposition}[theorem]{Proposition}
\newtheorem{corollary}[theorem]{Corollary}
\theoremstyle{definition}
\newtheorem{example}[theorem]{Example}
\newtheorem{remark}[theorem]{Remark}
\newtheorem*{remark*}{Remark}
\newcommand \KK {{\mathbb K}}
\newcommand \KKbar {\overline{\KK}}
\newcommand \ZZ {{\mathbb Z}}
\newcommand \RR {{\mathbb R}}
\newcommand \calA {{\mathcal{A}}}
\newcommand \calB {{\mathcal{B}}}
\newcommand\HH{{\mathcal{H}}}
\newcommand \calL {{\mathcal{L}}}
\newcommand \calO {{\mathcal{O}}}
\newcommand \calR {{\mathcal{R}}}
\newcommand \calT {{\mathcal{T}}}
\newcommand \bfq {{\bf{q}}}
\newcommand \bfqhat {\widehat{\bfq}}
\newcommand \Jbar {\overline{J}}
\newcommand \psibar {\overline{\psi}}
\newcommand \thetabar {\overline{\theta}}
\newcommand \rhotil {\widetilde{\rho}}
\newcommand \qhat {\widehat{q}}

\DeclareMathOperator \id { {\mathrm{id}} }
\DeclareMathOperator \rank { {\mathrm{rank}} }
\DeclareMathOperator \Span { {\mathrm{Span}} }
\DeclareMathOperator \charr { {\mathrm{char}} }
\DeclareMathOperator \Hom { {\mathrm{Hom}} }
\DeclareMathOperator \gr  { {\mathrm{gr}} }
\DeclareMathOperator \Spec {Spec}
\DeclareMathOperator \Prim {Prim}
\DeclareMathOperator\Sat{Sat}
\DeclareMathOperator \ann { {\mathrm{ann}} }
\DeclareMathOperator \ass { {\mathrm{ass}} }
\DeclareMathOperator\hgt{ht}
\newcommand\kx{\KK^\times}
\newcommand \Znn {\ZZ_{\ge 0}}
\newcommand \Zpos {\ZZ_{> 0}}
\newcommand\Abq{\calA_\bfq}
\newcommand\Tbq{\calT_\bfq}
\newcommand\longmapsfrom{\longleftarrow\!\shortmid}
\newcommand\Zq{Z(\Tbq)}

\begin{document}

\title[Binomial ideals in quantum tori and quantum affine spaces]{Binomial ideals in quantum tori\\ and quantum affine spaces}

\author[K. R. Goodearl]{K. R. Goodearl}
\address{
Department of Mathematics \\
University of California\\
Santa Barbara, CA 93106 \\
U.S.A.
}
\email{goodearl@math.ucsb.edu}

\begin{abstract}
The article targets binomial ideals in quantum tori and quantum affine spaces.  First, noncommutative analogs of known results for commutative (Laurent) polynomial rings are obtained, including the following:  Under the assumption of an algebraically closed base field, it is proved that primitive ideals are binomial, as are radicals of binomial ideals and prime ideals minimal over binomial ideals.  In the case of a quantum torus $\Tbq$, the results are strongest: In this situation, the binomial ideals are parametrized by characters on sublattices of the free abelian group whose group algebra is the center of $\Tbq$; the sublattice-character pairs corresponding to primitive ideals as well as to radicals and minimal primes of binomial ideals are determined. As for occurrences of binomial ideals in quantum algebras: It is shown that cocycle-twisted group algebras of finitely generated abelian groups are quotients of quantum tori modulo binomial ideals.  Another appearance is as follows: Cocycle-twisted semigroup algebras of finitely generated commutative monoids, as well as quantum affine toric varieties, are quotients of quantum affine spaces modulo certain types of binomial ideals.
\end{abstract}

\subjclass[2020]{16D25, 16D70, 16D80, 16N40, 16P40, 16S35, 16T20, 20G42}

\keywords{quantum torus, quantum affine space, binomial ideal, radical, minimal prime, associated prime, twisted group algebra}

\thanks{This research was supported
by US National Science Foundation grant DMS-1601184.}

\maketitle

\sectionnew{Introduction}

We undertake an intensive study of binomial ideals in quantum tori and quantum affine spaces; here \emph{binomials} are nonzero linear combinations of at most two monomials in the canonical generators (i.e., noncommutative polynomials with at most two terms), and \emph{binomial ideals} are ideals generated by binomials.  Many of the results are parallel to those known for commutative (Laurent) polynomial rings, even when not obtained by analogous methods.  We start by reviewing the commutative setting.
\medskip

Eisenbud and Sturmfels began their classic paper \cite{EiSt} by saying ``It is notoriously difficult to deduce anything about the structure of an ideal or scheme by directly examining its defining polynomials. A notable exception is that of monomial ideals."  They went on to show that binomial ideals are equally exceptional.  For instance, if $T = \KK[x_1^{\pm1}, \dots, x_n^{\pm1}]$ is a Laurent polynomial ring over a field $\KK$, the proper binomial ideals of $T$ are parametrized by pairs $(L,\rho)$ called \emph{partial characters of $\ZZ^n$}, meaning that $L$ is a sublattice of $\ZZ^n$ and $\rho$ a group homomorphism from $L$ to $\kx$.  We denote the corresponding binomial ideal of $T$ by 
\begin{equation}  \label{Ilrho.comm}
I(L,\rho) := \langle x^\alpha - \rho(\alpha) \mid \alpha \in L \rangle,
\end{equation}
where $x^\alpha$ is the monomial with multi-exponent $\alpha$.  If $\KK$ is algebraically closed, the prime ideals minimal over $I(L,\rho)$ are all binomial, of the form $I(L',\rho'_i)$ where $L'/L$ is the torsion subgroup of $\ZZ^n/L$ and the $\rho'_i$ are the extensions of $\rho$ to $L'$.  These minimal primes all have height equal to the rank of $L$.  Moreover, the radical of $I(L,\rho)$ is binomial.  If $\charr \KK = 0$, then $I(L,\rho)$ is its own radical; when $\charr \KK = p > 0$, on the other hand, $\sqrt{I(L,\rho)} = I(L'',\rho'')$ where $L''/L$ is the $p$-torsion subgroup of $\ZZ^n/L$ and $\rho''$ is the unique extension of $\rho$ to $L''$.

For the polynomial ring $\KK[x_1, \dots, x_n]$, there is no corresponding parametrization of all the binomial ideals.  However, if $\KK$ is algebraically closed, it is still true that minimal primes and radicals of binomial ideals are binomial.
\medskip

In this article, we develop an analogous binomial ideal theory in quantum tori $\Tbq$ and quantum affine spaces $\Abq$ over $\KK$, leading to noncommutative versions of the Eisenbud-Sturmfels results above.  Some of our results are carried over easily from those of \cite{EiSt} or are based on noncommutative analogs of lines in \cite{EiSt}; others are obtained via quite different methods.  

The reason why many results concerning Laurent polynomial rings $T = \KK[x_1^{\pm1}, \dots, x_n^{\pm1}]$ readily carry over to $\Tbq$ is due to the following known facts: The ideals of $\Tbq$ correspond bijectively to those of its center via contraction and extension, and the center $Z(\Tbq)$ is a Laurent polynomial ring over $\KK$, i.e., the group algebra $\KK \Gamma_Z$ of a free abelian group $\Gamma_Z$ of finite rank.  As we show, this correspondence restricts to bijections between the binomial ideals of $\Tbq$ and the binomial ideals of $Z(\Tbq)$.  In particular, the restricted correspondence allows us to parametrize the binomial ideals of $\Tbq$ by pairs $(L,\rho)$ where $L$ is a sublattice of $\Gamma_Z$ and $\rho \in \Hom(L,\kx)$.  The definition of the corresponding binomial ideals $I(L,\rho)$ in $\Tbq$ involves additional scalar terms compared with \eqref{Ilrho.comm}, however, due to the ambient noncommutativity of $\Tbq$ and the choice of polynomial variables in $Z(\Tbq)$.

We prove that
\begin{itemize}
\item \emph{If $\KK$ is algebraically closed, all primitive ideals of quantum tori and quantum affine spaces over $\KK$ are binomial ideals.} [Theorems \ref{maxTq} and \ref{primJ.Aq}]
\item \emph{If $\KK$ is perfect, the radical of any binomial ideal in a quantum torus or quantum affine space over $\KK$ is in turn binomial.} [Theorems \ref{rad.binom} and \ref{ES3.1}]
\item \emph{If $\KK$ is algebraically closed, the prime ideals minimal over any binomial ideal in a quantum torus or quantum affine space over $\KK$ are themselves binomial.} [Theorems \ref{minprimes.binom} and \ref{min.over.binom}]
\end{itemize}
In the quantum torus cases of these results, we characterize the partial characters that parametrize the primitive ideals and those that tag the radicals and minimal primes of binomial ideals.

Finally, we advertize several situations in which binomial ideals appear:
\begin{itemize}
\item \emph{The cocycle-twisted group algebras over $\KK$ of finitely generated abelian groups are precisely the quotients of quantum tori modulo proper binomial ideals.} [Theorem \ref{KcG=Tq/J}]
\item \emph{The cocycle-twisted semigroup algebras over $\KK$ of finitely generated commutative monoids are precisely the quotients of quantum affine spaces modulo binomial ideals that contain no monomials.} [Corollary \ref{KcS=Aq/B}]
\item \emph{Every quantum affine toric variety over $\KK$ is a quotient of a quantum affine space modulo a completely prime binomial ideal.  The converse also holds if $\KK$ is algebraically closed.} [Theorem \ref{qafftoric.as.Aq/B}]
\end{itemize}

\subsection{Some notation and conventions}
Fix an arbitrary base field $\KK$.

For any ring $R$, we write $\Spec R$ and $\Prim R$ for the respective sets of prime and primitive ideals of $R$, equipped with Zariski topologies. We write $\langle X \rangle \vartriangleleft R$ to denote that $\langle X \rangle$ is the ideal of $R$ generated by a set $X$. An element $a \in R$ is said to be \emph{normal} if $aR = Ra$, and it is \emph{regular} if it is a non-zero-divisor.  We say that $a$ is \emph{regular} (or \emph{normal}) \emph{modulo $I$ in $R$}, for an ideal $I$ of $R$, if the coset $a+I$ is regular (or normal) in $R/I$.  A \emph{regular sequence} in $R$ is any sequence of elements $a_1,\dots,a_r \in R$ such that $a_1$ is regular and $a_{i+1}$ is regular modulo $\langle a_1, \dots, a_i \rangle$ for $i \in [1,r-1]$.

The groups and semigroups appearing in this paper are all commutative, and will be written additively.

\sectionnew{Quantum affine spaces and quantum tori}  \label{notat}
We fix notation for a quantum affine space over $\KK$ and the corresponding quantum torus, then recall and extend well-known relations between the ideals of a quantum torus and its center.

\subsection{}
Let $\bfq = (q_{ij})$ be a multiplicatively skew-symmetric matrix in $M_n(\kx)$, meaning that $q_{ii} = 1$ and $q_{ji} = q_{ij}^{-1}$ for all $i,j \in [1,n]$. The corresponding (\emph{multiparameter}) \emph{quantum affine space} and \emph{quantum torus} are the algebras
\begin{align*}
\Abq &= \calO_\bfq(\KK^n) := \KK \langle x_1,\dots,x_n \mid x_i x_j = q_{ij} x_j x_i \ \forall\; i,j \in [1,n] \rangle  \\
\Tbq &= \calO_\bfq((\kx)^n) := \KK \langle x_1^{\pm1},\dots,x_n^{\pm1} \mid x_i x_j = q_{ij} x_j x_i \ \forall\; i,j \in [1,n] \rangle.
\end{align*}
As is well known, $\Abq$ and $\Tbq$ are noetherian domains (e.g., \cite[Theorem I.2.7, Corollary I.2.8]{BG}).

\subsection{}
Set $\Gamma := \ZZ^n$ and $\Gamma^+ := \Znn^n$.  We adopt standard multi-exponent notation for monomials in $\Abq$ and $\Tbq$.  Namely, for any $\alpha = (\alpha_1,\dots,\alpha_n) \in \Gamma$, we set
$$
x^\alpha := x_1^{\alpha_1} x_2^{\alpha_2} \cdots x_n^{\alpha_n}.
$$
The families $(x^\alpha \mid \alpha \in \Gamma^+)$ and $(x^\alpha \mid \alpha \in \Gamma)$ are $\KK$-bases for $\Abq$ and $\Tbq$, respectively.
These algebras are also graded by $\Gamma$, with $(\Tbq)_{\alpha} = \KK x^{\alpha}$ for $\alpha \in \Gamma$ and $(\Abq)_{\alpha} = \KK x^{\alpha}$ for $\alpha \in \Gamma^+$, while $(\Abq)_{\alpha} = 0$ for $\alpha \in \Gamma \setminus \Gamma^+$.

\subsection{}
Define $d : \Gamma\times\Gamma \rightarrow \kx$ by
\begin{equation}  \label{def.d}
d(\alpha,\beta) := \prod_{i=1}^n \prod_{j=1}^i q_{ij}^{\alpha_i \beta_j} = \prod_{i=2}^n \prod_{j=1}^{i-1} q_{ij}^{\alpha_i \beta_j},
\end{equation}
and note that $d$ is a homomorphism in each variable. This function is chosen so that
\begin{equation}  \label{d.prod}
x^\alpha x^\beta = d(\alpha,\beta) x^{\alpha+\beta} \qquad \forall\; \alpha,\beta \in \Gamma.
\end{equation}
Moreover, $d$ is a $2$-cocycle on $\Gamma$, and $\Tbq$, written in terms of the basis $( x^\alpha \mid \alpha \in \Gamma )$, coincides with the cocycle-twisted group algebra $\KK^d\Gamma$.

It follows from \eqref{d.prod} that
\begin{equation}  \label{yalpha.inv}
\bigl( x^\alpha \bigr)^{-1} = d(\alpha,\alpha) x^{- \alpha} \qquad \forall\; \alpha \in \Gamma.
\end{equation}

\subsection{} \label{ZTq.items}
Set $\Gamma_Z := \{ \alpha \in \Gamma \mid x^{\alpha} \in \Zq\}$, which is a sublattice (= subgroup) of $\Gamma$. Moreover, $\Zq$ is a homogeneous subalgebra of $\Tbq$ (with respect to the $\Gamma$-grading), and so
\begin{equation}  \label{ZTq}
\Zq = \bigoplus_{\alpha \in \Gamma_Z} \KK x^\alpha.
\end{equation}
Recall that a \emph{transversal} for a subgroup $H$ in an abelian group $G$ is a complete, irredundant set of representatives for the cosets of $H$ in $G$.
A basic observation is the following: If $S$ is a transversal for $\Gamma_Z$ in $\Gamma$, then $\Tbq$ is a free right or left $\Zq$-module with basis $(x^\sigma \mid \sigma \in S)$ (e.g., \cite[Proposition 2.3.3]{Kar}).

Clearly $\Tbq$ is graded-simple, and so its center is a Laurent polynomial ring of the form
$$\Zq = \KK [ (x^{\beta_1})^{\pm1}, \dots, (x^{\beta_r})^{\pm1} ],$$
where $(\beta_1,\dots,\beta_r)$ is a basis for $\Gamma_Z$ (e.g., \cite[Lemma II.3.7]{BG}), allowing $r=0$ in case $\Gamma_Z = 0$. Obviously $\Zq$ is isomorphic to the group algebra of $\Gamma_Z$, but a particular isomorphism will be needed. Fix the basis $(\beta_1,\dots,\beta_r)$ for $\Gamma_Z$, write $\KK\Gamma_Z$ as a Laurent polynomial ring of the form
$$\KK\Gamma_Z = \KK[ z_1^{\pm1}, \dots, z_r^{\pm1} ],$$
and then set
$$z_\gamma := z_1^{m_1} \cdots z_r^{m_r} \qquad \forall\; \gamma = m_1\beta_1 + \cdots + m_r\beta_r \in \Gamma_Z \,,$$
so that $(z_\gamma \mid \gamma \in \Gamma_Z)$ is a $\KK$-basis for $\KK\Gamma_Z$ and $z_\gamma z_\delta = z_{\gamma+\delta}$ for $\gamma,\delta \in \Gamma_Z$.

There is a $\KK$-algebra isomorphism 
\begin{equation}  \label{def.phi}
\phi : \KK\Gamma_Z \overset{\cong}{\longrightarrow} \Zq \quad\text{such that}\quad \phi(z_i) = x^{\beta_i}\ \ \forall\, i \in [ 1,r ].
\end{equation}
Observe that there are scalars $c(\gamma) \in \kx$ such that
\begin{equation}  \label{def.c}
\phi(z_\gamma) = c(\gamma) x^\gamma \qquad \forall\; \gamma \in \Gamma_Z \,.
\end{equation}
It follows from \eqref{d.prod} that
\begin{equation}  \label{d.ab}
d(\alpha,\beta) = c(\alpha+\beta) c(\alpha)^{-1} c(\beta)^{-1} \qquad \forall\, \alpha, \beta \in \Gamma_Z \,.
\end{equation}

\subsection{} \label{ideals.of.Tq}
The ideals of $\Zq$ and $\KK\Gamma_Z$ match up with those of $\Tbq$ as follows. Note that if we also use the name $\phi$ for the map with codomain enlarged to $\Tbq$, then the set map $J \mapsto \phi^{-1}(J)$ provides a bijection from the set of ideals of $\Tbq$ onto the set of ideals of $\KK\Gamma_Z$ and restricts to a bijection of $\Spec \Tbq$ onto $\Spec \KK\Gamma_Z$.

\begin{proposition}  \label{idealsTq}
There are inverse inclusion-preserving bijections
\begin{equation}   \label{ideals.biject}
\begin{array}{r c l r c l}
\{ \,\text{ideals of}\ \KK\Gamma_Z \,\} &\longleftrightarrow &\multicolumn{2}{c}{\{ \,\text{ideals of}\ \Zq \,\}} &\longleftrightarrow &\{ \,\text{ideals of}\ \Tbq \, \}  \\
I_0 &\longmapsto &\phi(I_0) &I &\longmapsto &I\cdot \Tbq  \\
\phi^{-1}(J_0) &\longmapsfrom &J_0 &J \cap \Zq &\longmapsfrom &J.
\end{array}
\end{equation}
These bijections restrict to bijections
\begin{equation}  \label{primes.biject}
\Spec \KK\Gamma_Z \longleftrightarrow \Spec \Zq \longleftrightarrow \Spec \Tbq \,.
\end{equation}
\end{proposition}

\begin{proof} The left hand bijections in \eqref{ideals.biject} are immediate from \eqref{def.phi}, while the right hand ones are given by \cite[Proposition II.3.8]{BG}, for instance. The bijections \eqref{primes.biject} follow easily \cite[Exercise II.3.C]{BG}.
\end{proof}

The \emph{height} of a prime ideal $P$ in a ring $R$, denoted $\hgt P$, is the supremum of the lengths $r$ of chains $P_0 = P \supsetneqq P_1 \supsetneqq \cdots \supsetneqq P_r$ of prime ideals in $R$. Following commutative usage,  the \emph{height} of an arbitrary ideal $I$ of $R$ is the infimum of the heights of the prime ideals containing $I$. 

\begin{corollary}  \label{hts.Tq}
Let $I$ be an ideal of $\Tbq$.

{\rm(a)} The bijections \eqref{ideals.biject} restrict to bijections between the sets
\begin{itemize}
\item $\{\, \text{prime ideals of}\; \KK\Gamma_Z\; \text{minimal over}\; \phi^{-1}(I)\, \}$,
\item $\{\, \text{prime ideals of}\; Z(\Tbq)\; \text{minimal over}\; I \cap Z(\Tbq)\, \}$,
\item $\{\, \text{prime ideals of}\; \Tbq\; \text{minimal over}\; I\, \}$.
\end{itemize}

{\rm(b)} $\hgt I = \hgt (I\cap Z(\Tbq)) = \hgt \phi^{-1}(I)$.

{\rm(c)} If $z \in Z(\Tbq)$, then $z$ is regular modulo $I$ in $\Tbq$ if and only if $z$ is regular modulo $I \cap Z(\Tbq)$ in $Z(\Tbq)$.
\end{corollary}

\begin{proof}
(a) and (b) follow directly from Proposition \ref{idealsTq}.

(c) Set $J := \{ t \in \Tbq \mid zt \in I \}$ and $J_Z := \{ y \in Z(\Tbq) \mid zy \in I_Z \}$, where $I_Z := I \cap Z(\Tbq)$.  Then $J$ and $J_Z$ are ideals of $\Tbq$ and $Z(\Tbq)$, resp., and $J_Z = J \cap Z(\Tbq)$.  Now $z$ is regular modulo $I$ in $\Tbq$ if and only if $J = I$, and $z$ is regular modulo $I_Z$ in $Z(\Tbq)$ if and only if $J_Z = I_Z$.  Since $J = I$ if and only if $J_Z = I_Z$, part (c) follows.
\end{proof}

We recall the concept of associated primes for modules $M$ over a noncommutative noetherian ring $R$ (e.g., \cite[p.56]{GW}, \cite[\S\S4.3.4, 4.4.4, Lemma 4.4.5]{MR}): a prime ideal $P$ of $R$ is an \emph{associated prime of $M$} provided $M$ has a submodule $M'$ which is a \emph{fully faithful $(R/P)$-module}, meaning that $P = \ann(M'')$ for all nonzero submodules $M'' \le M'$. The set of associated primes of $M$ is denoted $\ass(M)$.

\begin{lemma}  \label{ass.M.M0}
Let $M_0$ be a left $\Zq$-module and $M$ the left $\Tbq$-module $\Tbq \otimes_{\Zq} M_0$. Let $P \in \Spec \Tbq$ and $P_0 := P \cap \Zq \in \Spec \Zq$. The following are equivalent:

{\rm(a)} $P \in \ass(M)$.

{\rm(b)} $P_0 \in \ass(M_0)$.

{\rm(c)} $P = \ann_{\Tbq}(\Tbq a)$ for some $a\in M$.
\end{lemma}

\begin{proof}
Let $T$ be a transversal for $\Gamma_Z$ in $\Gamma$. As noted in \S\ref{ZTq.items}, $\Tbq$ is a free right $\Zq$-module with basis $(x^\tau \mid \tau \in T)$. Consequently, $M = \bigoplus_{\tau\in T} x^\tau \otimes M_0$ as abelian groups.

(a)$\Longrightarrow$(c): There is a submodule $M' \le M$ which is a fully faithful $(\Tbq/P)$-module. Choose $0 \ne a \in M'$; then $P = \ann_{\Tbq}(\Tbq a)$.

(c)$\Longrightarrow$(b): Write $a = \sum_{\tau\in T'} x^\tau \otimes a_\tau$ for some finite subset $T' \subseteq T$ and some $a_\tau \in M_0$. For any $z \in \Zq$, we have $z \in P$ iff $za = 0$ iff $\sum_{\tau\in T'} x^\tau \otimes za_\tau = 0$ iff $za_\tau = 0$ for all $\tau \in T'$, due to the centrality of $z$.  Thus,
$$P_0 = \bigcap_{\tau \in T'} \ann_{\Zq}(a_\tau).$$
Since $P_0$ is prime, $P_0 = \ann_{\Zq}(a_\sigma)$ for some $\sigma \in T'$, and (b) follows.

(b)$\Longrightarrow$(a): There exists $b \in M_0$ such that $P_0 = \ann_{\Zq}(b)$. Let $b' := 1\otimes b \in M$ and $P' := \ann_{\Tbq}(\Tbq b')$. As above, $P' \cap \Zq = P_0$, and consequently $P' = P$.

We will be done if $\Tbq b'$ is a fully faithful $(\Tbq/P)$-module, i.e., if $P = \ann_{\Tbq}(M')$ for any nonzero submodule $M' \le \Tbq b'$. Choose $Q \in \ass(M')$ (which exists because $\Tbq$ is noetherian), and set $Q_0 := Q \cap \Zq$. There is a nonzero submodule $M'' \le M'$ which is a fully faithful $(\Tbq/Q)$-module. Choose $c \in \Tbq$ such that $0 \ne cb' \in M''$, so that $\ann_{\Tbq}(\Tbq cb') = Q$. The argument of (c)$\Longrightarrow$(b) shows that $Q_0 = \ann_{\Zq}(c_\sigma b)$ for some $c_\sigma \in Z(\Tbq)$.
Now $c_\sigma b \ne 0$ and $Q_0 = P_0$ since $\Zq b \cong \Zq/P_0$. Consequently,
$$P = Q = \ann_{\Tbq}(M'') \supseteq \ann_{\Tbq}(M') \supseteq \ann_{\Tbq}(\Tbq b') = P' = P,$$
and therefore $P = \ann_{\Tbq}(M')$ as desired.
\end{proof}

\begin{corollary}  \label{ass.Tq}
Let $I$ be an ideal of $\Tbq$. The bijections \eqref{ideals.biject} restrict to bijections between the sets $\ass({}_{\KK\Gamma_Z}(\KK\Gamma_Z/\phi^{-1}(I)))$, $\ass({}_{Z(\Tbq)}(Z(\Tbq)/(I \cap Z(\Tbq))))$, and $\ass({}_{\Tbq}(\Tbq/I))$.
\end{corollary}

\begin{proof}
The bijection between the first and second sets occurs because $\phi$ is an isomorphism of $\KK\Gamma_Z$ onto $Z(\Tbq)$ that carries $\phi^{-1}(I)$ onto $I \cap Z(\Tbq)$. The bijection between the second and third sets follows from Proposition \ref{idealsTq} and the case $M_0 = Z(\Tbq)/(I \cap Z(\Tbq))$ of Lemma \ref{ass.M.M0}.
\end{proof}

\sectionnew{Binomial ideals in $\Tbq$}  \label{binomTq}
We first show that binomial ideals of the quantum torus $\Tbq$ correspond to binomial ideals of its center, and that quotients modulo these binomial ideals are twisted group algebras. The connection with binomial ideals of the center allows the classical results of Eisenbud and Sturmfels \cite{EiSt} on radicals, minimal and associated primes of binomial ideals in Laurent polynomial rings to be carried over to $\Tbq$. Some of these results also follow from twisted group algebra structures on quotient algebras.

\subsection{}
Following \cite[\S2]{EiSt}, we define a (\emph{Laurent}) \emph{binomial} in $\KK\Gamma_Z$ to be any nonzero element of the form $\lambda z_\alpha + \mu z_\beta$ where $\lambda, \mu \in \KK$ and $\alpha, \beta \in \Gamma$, and we define a (\emph{Laurent}) \emph{binomial ideal} in $\KK\Gamma_Z$ to be any ideal generated by binomials. By \cite[Theorem 2.1(a)]{EiSt}, the binomial ideals of $\KK\Gamma_Z$ can be described in terms of \emph{partial characters} of $\Gamma_Z$, meaning homomorphisms from sublattices of $\Gamma_Z$ to $\kx$. Allowing empty sets of generators, the zero ideal is a binomial ideal of $\KK\Gamma_Z$. Moreover, $\KK\Gamma_Z = \langle z_0 \rangle$ is a binomial ideal.

Carrying this terminology over to quantum tori, let us say that a (\emph{Laurent}) \emph{binomial} in $\Tbq$ is any nonzero element of the form $\lambda x^\alpha + \mu x^\beta$ where $\lambda, \mu \in \KK$ and $\alpha, \beta \in \Gamma$, and that a (\emph{Laurent}) \emph{binomial ideal} in $\Tbq$ is any ideal generated by binomials.  Note that $0$ and $\Tbq$ are binomial ideals of $\Tbq$.

\begin{lemma}  \label{binom.biject}
The bijections \eqref{ideals.biject} restrict to bijections
$$\{ \,\text{binomial ideals of}\ \KK\Gamma_Z \,\} \longleftrightarrow \{ \,\text{binomial ideals of}\ \Tbq \, \}.$$
\end{lemma}

\begin{proof} Since $\phi$ maps binomials $\lambda z_\alpha + \mu z_\beta$ in $\KK\Gamma_Z$ to binomials $\lambda c(\alpha) x^\alpha + \mu c(\beta) x^\beta$ in $\Tbq$, it is clear that if $I$ is a binomial ideal in $\KK\Gamma_Z$, then $\phi(I) \cdot \Tbq$ is a binomial ideal of $\Tbq$.

Now let $J$ be an arbitrary binomial ideal of $\Tbq$ and $I := \phi^{-1}(J)$. We must show that $I$ is a binomial ideal of $\KK\Gamma_Z$. This is clear if $J = \Tbq$, since then $I = \KK\Gamma_Z$. Hence, assume that $J$ is a proper ideal of $\Tbq$.

Let  $\{ \lambda_s x^{\alpha_s} + \mu_s x^{\beta_s} \mid s \in S \}$
be a set of binomial generators for $J$, for some $\lambda_s, \mu_s \in \KK$, not both zero, and some $\alpha_s, \beta_s \in \Gamma$. If one of $\lambda_s$ or $\mu_s$ is zero, then $\lambda_s x^{\alpha_s} + \mu_s x^{\beta_s}$ is a unit in $\Tbq$, which would force $I = \Tbq$. Thus, $\lambda_s, \mu_s \ne 0$ for $s \in S$. Similarly, $\alpha_s \ne \beta_s$ for $s \in S$.

For $s \in S$, we have
\begin{align*}
x^{\alpha_s} (x^{\beta_s})^{-1} &= x^{\alpha_s} d(\beta_s,\beta_s) x^{-\beta_s} = d(\beta_s,\beta_s) d(\alpha_s, -\beta_s)  x^{\alpha_s - \beta_s} = d(\beta_s - \alpha_s, \beta_s) x^{\alpha_s - \beta_s}  \\
\lambda_s x^{\alpha_s} + \mu_s x^{\beta_s} &= \bigl( \lambda_s x^{\alpha_s} (x^{\beta_s})^{-1} + \mu_s \bigr) x^{\beta_s} = \lambda_s d(\beta_s - \alpha_s, \beta_s) \bigl( x^{\alpha_s - \beta_s} - \nu_s\bigr) x^{\beta_s}
\end{align*}
where $\nu_s := - \lambda_s^{-1} d(\alpha_s - \beta_s, \beta_s) \mu_s$,
and so the generator $\lambda_s x^{\alpha_s} + \mu_s x^{\beta_s}$ may be replaced by $x^{\alpha_s - \beta_s} - \nu_s$. Consequently $J$ has a set of generators of the form $\{ x^{\gamma_s} - \nu_s \mid s \in S \}$, for some $\gamma_s \in \Gamma$ and $\nu_s \in \kx$.

For $s \in S$ and $i \in [ 1,n ]$, there is some $p_{si} \in \kx$ such that $x^{\gamma_s} x_i = p_{si} x_i x^{\gamma_s}$, and so $J$ contains the element
$$(x^{\gamma_s} - \nu_s) x_i - p_{si} x_i (x^{\gamma_s} - \nu_s) = (p_{si}-1) \nu_s x_i \,.$$
Since $J$ is proper, all the $p_{si} = 1$, whence $x^{\gamma_s} \in \Zq$. Thus $\gamma_s \in \Gamma_Z$ for all $s \in S$.

Let $I'$ be the ideal of $\KK\Gamma_Z$ generated by $\{ z_{\gamma_s} - c(\gamma_s) \nu_s \mid s\in S \}$. Since $\phi( z_{\gamma_s} - c(\gamma_s) \nu_s ) = c(\gamma_s) ( x^{\gamma_s} - \nu_s )$, we see that $\phi(I')$ is the ideal of $\Zq$ generated by $\{ x^{\gamma_s} - \nu_s \mid s \in S \}$. Consequently, $\phi(I') \cdot \Tbq = J$, whence $I' = \phi^{-1}(J) = I$ by Proposition \ref{idealsTq}.
Therefore $I$ is a binomial ideal of $\KK\Gamma_Z$, as required.
\end{proof}

\begin{theorem}  \label{binom.Tq}
{\rm(a)} Let $L$ be a sublattice of $\Gamma_Z$ and $\rho \in \Hom(L,\kx)$. Then
\begin{equation}  \label{def.ILrho}
\begin{aligned}
I_0(L,\rho) &:= \langle z_\alpha - \rho(\alpha) \mid \alpha \in L \rangle \vartriangleleft \KK\Gamma_Z  \\
I(L,\rho) &:= \langle x^\alpha - c(\alpha)^{-1} \rho(\alpha) \mid \alpha \in L \rangle \vartriangleleft \Tbq
\end{aligned}
\end{equation}
are proper binomial ideals of $\KK\Gamma_Z$ and $\Tbq$, respectively, and $I(L,\rho) = \phi( I_0(L,\rho) ) \cdot \Tbq$.

{\rm(b)} Let $J$ be a proper binomial ideal of $\Tbq$ and 
$$
L := \{ \alpha \in \Gamma \mid x^\alpha \in \kx + J \}.
$$
Then $L$ is a sublattice of $\Gamma_Z$ and there is a unique $\rho \in \Hom(L,\kx)$ such that $J = I(L,\rho)$.
\end{theorem}

\begin{proof}
(a) Observe that the set $\{ z_\alpha - \rho(\alpha) \mid \alpha \in L \}$ generates a maximal ideal $M$ of $\KK L$ and $M \cdot \KK\Gamma_Z = I_0(L,\rho)$. Since 
\begin{equation}  \label{phi.I0Lrho}
\phi( z_\alpha - \rho(\alpha) ) = c(\alpha) (x^\alpha - c(\alpha)^{-1} \rho(\alpha)) \qquad \forall\,\alpha \in L,
\end{equation}
we see that 
$I(L,\rho) = \phi(M) \cdot \Tbq = \phi( I_0(L,\rho) ) \cdot \Tbq$.
Since $\KK\Gamma_Z$ is a free left $\KK L$-module (e.g., \cite[Lemma 1.1.3]{Pas}), we have
$$\KK\Gamma_Z / I_0(L,\rho) \cong (\KK L / M) \otimes_{\KK L} \KK\Gamma_Z \ne 0,$$
and so $I_0(L,\rho)$ is a proper ideal of $\KK\Gamma_Z$. It then follows from Proposition \ref{idealsTq} that $I(L,\rho)$ is a proper ideal of $\Tbq$.

(b) For $\alpha \in L$, there exists $\rho'(\alpha) \in \kx$ such that $x^\alpha - \rho'(\alpha) \in J$, and $\rho'(\alpha)$ is unique because $J$ is proper. For $\alpha,\beta \in L$, we have
$$
x^\alpha x^\beta - \rho'(\alpha) \rho'(\beta) = (x^\alpha - \rho'(\alpha)) x^\beta + \rho'(\alpha) (x^\beta - \rho'(\beta)) \in J,
$$
whence $x^{\alpha+\beta} - d(\alpha,\beta)^{-1} \rho'(\alpha) \rho'(\beta) \in J$. Consequently, 
\begin{equation}  \label{L,+}
\alpha+\beta \in L \qquad\text{and}\qquad \rho'(\alpha+\beta) = d(\alpha,\beta)^{-1} \rho'(\alpha) \rho'(\beta) \qquad \forall\; \alpha,\beta \in L.
\end{equation}
Since we also have $x^0 - 1 \in J$ and 
$$
x^{-\alpha} - \rho'(\alpha)^{-1} d(-\alpha, \alpha) = - \rho'(\alpha)^{-1} x^{-\alpha} (x^\alpha - \rho'(\alpha)) \in J \qquad \forall\; \alpha \in L,
$$
we find that $L$ is a sublattice of $\Gamma$. The argument of Lemma \ref{binom.biject} shows that $L \subseteq \Gamma_Z$.

Set $\rho(\alpha) := c(\alpha) \rho'(\alpha)$ for $\alpha \in L$. In view of \eqref{L,+} and \eqref{d.ab},
$$
\rho(\alpha+\beta) = c(\alpha+\beta) d(\alpha,\beta)^{-1} \rho'(\alpha) \rho'(\beta) = c(\alpha) c(\beta) \rho'(\alpha) \rho'(\beta) = \rho(\alpha) \rho(\beta)
$$
for $\alpha, \beta \in L$. Thus, $\rho \in \Hom(L,\kx)$.

By part (a), $I(L,\rho)$ is a proper binomial ideal of $\Tbq$, and $I(L,\rho) \subseteq J$ by definition of $\rho$ and $\rho'$. The proof of Lemma \ref{binom.biject} shows that $J$ is generated by a set $\{ x^{\gamma_s} - \nu_s \mid s \in S \}$ of binomials, for some $\gamma_s \in \Gamma$ and $\nu_s \in \kx$. Each $\gamma_s \in L$ and $\nu_s = \rho'(\gamma_s)$, whence $x^{\gamma_s} - \nu_s = x^{\gamma_s} - c(\gamma_s)^{-1} \rho(\gamma_s) \in I(L,\rho)$. Therefore $J = I(L,\rho)$.

Concerning the uniqueness of $\rho$, suppose that $J = I(L,\sigma)$ for some $\sigma \in \Hom(L,\kx)$. For $\alpha \in L$, we then have $x^\alpha - c(\alpha)^{-1} \sigma(\alpha) \in J$, whence $\rho'(\alpha) = c(\alpha)^{-1} \sigma(\alpha)$ and $\rho(\alpha) = \sigma(\alpha)$.  Therefore $\rho = \sigma$.
\end{proof}

The processes described in Theorem \ref{binom.Tq} actually give inverse bijections, as we show in Theorem \ref{biject.pairs.binom.Tq}.

\begin{remark}  \label{avoid.c}
The appearance of the scalars $c(\alpha)$ in the definition of $I(L,\rho)$ is due to the choice of a splitting of the subgroup $\kx x^L := \{ \lambda x^\alpha \mid \lambda \in \kx,\; \alpha \in L \}$ of $U(\Zq)$. Namely, there is a surjective group homomorphism $\kx x^L \rightarrow L$ given by $\lambda x^\alpha \mapsto \alpha$, which has a right inverse $\xi : L \rightarrow \kx x^L$ given by $\alpha \mapsto c(\alpha) x^\alpha$. It follows from \eqref{d.prod} and \eqref{d.ab} that $\xi$ is a homomorphism. Transferring $\rho$ to $\xi(L)$ via $\xi$ and using the splitting $\kx x^L = \kx \times \xi(L)$, we obtain a homomorphism $\rhotil : \kx x^L \rightarrow \kx$ such that $\rhotil|_{\kx} = \id_{\kx}$ and $\rhotil(x^\alpha) = c(\alpha)^{-1} \rho(\alpha)$ for $\alpha \in L$.

Thus, the proper binomial ideals of $\Tbq$ may be expressed in the form
$$
\langle u - \rhotil(u) \mid u \in \kx x^L \rangle
$$
for sublattices $L$ of $\Gamma_Z$ and $\rhotil \in \Hom(\kx x^L, \kx)$ such that $\rhotil|_{\kx} = \id_{\kx}$. This allows these ideals to be expressed independently of the choice of the isomorphism $\phi$ in \eqref{def.phi}.
\end{remark}

Since $\Tbq$ is noetherian, each $I(L,\rho)$ is finitely generated. Such an ideal can be generated by binomials corresponding to a basis for $L$, as follows.

\begin{lemma}  \label{ILrho.basis}
Let $L$ be a sublattice of $\Gamma_Z$ and $\rho \in \Hom(L,\kx)$. If $(\beta_1,\dots,\beta_r)$ is a basis for $L$, then $I_0(L,\rho)$ is generated by $\{ z_{\beta_l} - \rho(\beta_l) \mid l \in [1,r] \}$ and $I(L,\rho)$ is generated by $\{ x^{\beta_l} - c(\beta_l)^{-1} \rho(\beta_l) \mid l \in [1,r] \}$.  Moreover,
\begin{equation}  \label{reg.seq.ILrho}
x^{\beta_1} - c(\beta_1)^{-1} \rho(\beta_1),\, \dots,\, x^{\beta_r} - c(\beta_r)^{-1} \rho(\beta_r)
\end{equation}
is a regular sequence in $\Tbq$.
\end{lemma}

\begin{proof}
For the first part, it suffices, in view of Theorem \ref{binom.Tq} and equation \eqref{phi.I0Lrho}, to prove that $I_0(L,\rho)$ equals the ideal 
$$I'_0 := \langle z_{\beta_l} - \rho(\beta_l) \mid l \in [1,r] \rangle \vartriangleleft \KK \Gamma_Z\,.$$
Observe that the set $L' := \{ \alpha \in L \mid z_\alpha - \rho(\alpha) \in I'_0 \}$ is a sublattice of $L$ containing the $\beta_l$. Therefore $L' = L$, whence $I'_0 = I_0(L,\rho)$.

By \cite[Theorem 2.1(b)]{EiSt}, $z_{\beta_1} - \rho(\beta_1), \dots, z_{\beta_r} - \rho(\beta_r)$ is a regular sequence in $\KK \Gamma_Z$, and so it follows from \eqref{phi.I0Lrho} that \eqref{reg.seq.ILrho} is a regular sequence in $Z(\Tbq)$. Proposition \ref{idealsTq} and Corollary \ref{hts.Tq}(c) thus imply that \eqref{reg.seq.ILrho} is a regular sequence in $\Tbq$.
\end{proof}

Geometrically, \cite[Theorem 2.1(b)]{EiSt} says that $\KK \Gamma_Z / I_0(L,\rho)$ is (the coordinate ring of) an affine complete intersection.  One could thus say that Lemma \ref{ILrho.basis} shows that $\Tbq/I(L,\rho)$ is a noncommutative affine complete intersection.

Quotients of $\Tbq$ and $\Zq$ by proper binomial ideals may be expressed in terms of (cocycle-)twisted group algebras, as follows. By \cite[Proposition 1.4.2]{NVO}, for instance, a $\KK$-algebra $T$ is a twisted group algebra of an additive group $G$ over $\KK$ if and only if $T$ has a $\KK$-basis $(t_g \mid g \in G)$ such that $t_g t_h \in \kx t_{g+h}$ for all $g,h \in G$.

\begin{proposition}  \label{Irho.ideals}
Let $L$ be a sublattice of $\Gamma_Z$ and $\rho \in \Hom(L,\kx)$, and set
\begin{equation}  \label{def.IZLrho}
I_Z(L,\rho) := \langle x^\alpha - c(\alpha)^{-1} \rho(\alpha) \mid \alpha \in L \rangle \vartriangleleft \Zq.
\end{equation}

{\rm(a)} $I_Z(L,\rho) = I(L,\rho) \cap \Zq$.

{\rm(b)} $\Zq/I_Z(L,\rho)$ is a twisted group algebra of $\Gamma_Z/L$ over $\KK$, with a $\KK$-basis of the form $( x^\tau + I_Z(L,\rho) \mid \tau \in T )$ where $T$ is any transversal for $L$ in $\Gamma_Z$.

{\rm(c)} $\Tbq/I(L,\rho)$ is a twisted group algebra of $\Gamma/L$ over $\KK$, with a $\KK$-basis of the form $( x^\sigma + I(L,\rho) \mid \sigma \in S )$ where $S$ is any transversal for $L$ in $\Gamma$.
\end{proposition}

\begin{proof} 
(a) Observe that since $I(L,\rho) = I_Z(L,\rho) \cdot \Tbq$, it follows from Proposition \ref{idealsTq} that  $I(L,\rho) \cap \Zq = I_Z(L,\rho)$.

(b) As in the proof of Theorem \ref{binom.Tq}, the binomials $z_\alpha - \rho(\alpha)$ for $\alpha \in L$ generate a maximal ideal $M$ of $\KK L$ (of codimension $1$) and $\phi(M) \cdot \Zq = I_Z(L,\rho)$. By \cite[Lemma 1.2.2 and proof]{Pas}, $\KK\Gamma_Z / M\cdot \KK\Gamma_Z$ is a twisted group algebra of $\Gamma_Z/L$ over $\KK L/M = \KK$ and $(z_\tau + M\cdot \KK\Gamma_Z \mid \tau \in T)$ is a $\KK$-basis for $\KK\Gamma_Z / M\cdot \KK\Gamma_Z$. Part (b) follows.

(c) Let $T$ be a transversal for $L$ in $\Gamma_Z$ and $S'$ a transversal for $\Gamma_Z$ in $\Gamma$, and note that $S := S' + T$ is a transversal for $L$ in $\Gamma$. As noted in \S\ref{ZTq.items}, $\Tbq$ is a free left $\Zq$-module with basis $(x^\sigma \mid \sigma \in S')$. Since
$$\Tbq/I(L,\rho) = \Tbq/ I_Z(L,\rho) \cdot \Tbq \cong \bigl( \Zq/ I_Z(L,\rho) \bigr) \otimes_{\Zq} \Tbq \,,$$
it follows that 
$$( x^\sigma + I(L,\rho) \mid \sigma \in S ) = ( d(\sigma',\tau)^{-1} x^{\sigma'} x^\tau + I(L,\rho) \mid \sigma' \in S',\; \tau \in T )$$
is a $\KK$-basis for $\Tbq/I(L,\rho)$. 

Suppose that $S_1$ is another transversal for $L$ in $\Gamma$. There is a bijection $\pi : S_1 \rightarrow S$ such that $\sigma+L = \pi(\sigma)+L$ for all $\sigma \in S_1$. Given any $\sigma \in S_1$, we have $\alpha := \sigma - \pi(\sigma) \in L$ and so
$$x^\sigma = x^{\pi(\sigma)+\alpha} = d(\pi(\sigma),\alpha)^{-1} x^{\pi(\sigma)} x^\alpha \equiv d(\pi(\sigma),\alpha)^{-1} c(\alpha)^{-1} \rho(\alpha) x^{\pi(\sigma)} \pmod{I(L,\rho)},$$
whence $x^\sigma + I(L,\rho) \in \kx (x^{\pi(\sigma)} + I(L,\rho) ).$ Consequently, $( x^\sigma + I(L,\rho) \mid \sigma \in S_1 )$ is also a $\KK$-basis for $\Tbq/I(L,\rho)$.

Set $[\sigma] := \sigma + L$ and $x_{[\sigma]} := x^\sigma + I(L,\rho)$ for $\sigma \in S$, so that $\sigma \mapsto [\sigma]$ is a bijection from $S$ onto $\Gamma/L$ and $( x_{[\sigma]} \mid \sigma \in S )$ is a $\KK$-basis for $\Tbq/I(L,\rho)$. For any $\sigma, \sigma' \in S$, we have $\sigma + \sigma' = \gamma + \alpha$ for some $\gamma \in S$ and $\alpha \in L$, whence
\begin{align*}
x_{[\sigma]} x_{[\sigma']} &= d(\sigma,\sigma') x^{\sigma+\sigma'} + I(L,\rho) = d(\sigma,\sigma') d(\gamma,\alpha)^{-1} x^\gamma x^\alpha + I(L,\rho)  \\
&= d(\sigma,\sigma') d(\gamma,\alpha)^{-1} c(\alpha)^{-1} \rho(\alpha) x^\gamma + I(L,\rho)  \\
&= d(\sigma,\sigma') d(\gamma,\alpha)^{-1} c(\alpha)^{-1} \rho(\alpha) x_{[\sigma]+[\sigma']}\,.
\end{align*}
It follows that $\Tbq/I(L,\rho)$ is a twisted group algebra of $\Gamma/L$ over $\KK$.
\end{proof}

\begin{theorem}  \label{KcG=Tq/J}
Let $R$ be a unital $\KK$-algebra. Then $R$ is isomorphic to a twisted group algebra $\KK^eG$, for some finitely generated abelian group $G$ and some $2$-cocycle $e : G \times\nobreak G \rightarrow \kx$, if and only if $R \cong \Tbq/J$ for some multiplicatively skew-symmetric matrix $\bfq \in M_n(\kx)$, some $n \in \Zpos$, and some proper binomial ideal $J$ of $\Tbq$.
\end{theorem}

\begin{proof}
$(\Longleftarrow)$: Theorem \ref{binom.Tq}(b) and Proposition \ref{Irho.ideals}(c).

$(\Longrightarrow)$: By assumption, $R$ has a $\KK$-basis $(y_g \mid g \in G)$ such that the $y_g$ are units in $R$ and $y_g y_h = e(g,h) y_{g+h}$ for $g,h \in G$. From the case $g=h=0$, it follows that $y_0 = e(0,0)\cdot1_R$. Let $\{g_1,\dots,g_n\}$ be a finite set of generators for $G$, and set $z_i := y_{g_i}$ and $q_{ij} := e(g_i,g_j) e(g_j,g_i)^{-1}$ for $i,j \in [1,n]$. Then $\bfq := (q_{ij}) \in M_n(\kx)$ is a multiplicatively skew-symmetric matrix and $z_i z_j = q_{ij} z_j z_i$ for all $i$, $j$.  Now there is a unital surjective $\KK$-algebra homomorphism $\theta : \Tbq \rightarrow R$ such that $\theta(x_i) = z_i$ for all $i$.  Note that $\ker \theta$ is a proper ideal of $\Tbq$.

Let $\pi : \Gamma \rightarrow G$ be the group homomorphism given by the rule $\pi(\alpha) = g_1^{\alpha_1} \cdots g_n^{\alpha_n}$. For each $\alpha \in \Gamma$, we have
$$
\theta(x^\alpha) = \theta( x_1^{\alpha_1} \cdots x_n^{\alpha_n} ) = y_{g_1}^{\alpha_1} \cdots y_{g_n}^{\alpha_n} = \mu(\alpha) y_{\pi(\alpha)} \qquad \text{for some}\; \mu(\alpha) \in \kx.
$$
In particular, for $\alpha \in \ker \pi$ we have $\theta(x^\alpha) = \mu(\alpha) y_0 = \mu(\alpha) \mu(0)^{-1} \theta(1)$, and so $\ker \theta$ contains the binomial ideal
$$
J := \langle x^\alpha - \mu(\alpha) \mu(0)^{-1} \mid \alpha \in \ker \pi \rangle \vartriangleleft \Tbq \,.
$$
By Theorem \ref{binom.Tq}(b), $J = I(L,\rho)$ where $L := \{ \alpha \in \Gamma \mid x^\alpha \in \kx+J \}$ is a sublattice of $\Gamma_Z$ and $\rho \in \Hom(L,\kx)$. Note that $\ker \pi \subseteq L$ by definition of $J$.

Let $S$ be a transversal for $L$ in $\Gamma$, and enlarge $S$ to a transversal $T$ for $\ker \pi$ in $\Gamma$. Since $\pi$ maps $T$ bijectively onto $G$, the map $\theta$ sends the family $(x^\tau \mid \tau \in T)$ to a $\KK$-basis for $R$, and so $(x^\tau \mid \tau \in T)$ is $\KK$-linearly independent module $\ker \theta$, hence also $\KK$-linearly independent modulo $J$. On the other hand, Proposition \ref{Irho.ideals}(c) shows that $(x^\sigma \mid \sigma \in S)$ is a $\KK$-basis for $\Tbq/J$. This forces $T = S$, and so the induced homomorphism $\thetabar : \Tbq/J \rightarrow R$ sends a $\KK$-basis for $\Tbq/J$ to a $\KK$-basis for $R$. Consequently, $\thetabar$ is an isomorphism, and therefore $\ker \theta = J$.
\end{proof}

We use the results of Proposition \ref{Irho.ideals} to complete the bijection initiated in Theorem \ref{binom.Tq}, as follows.

\begin{theorem}  \label{biject.pairs.binom.Tq}
Let $\calL$ denote the set of pairs $(L,\rho)$ where $L$ is a sublattice of $\Gamma_Z$ and $\rho \in \Hom(L,\kx)$. Then there exists a bijection
\begin{align*}
I : \calL &\longrightarrow \{\;\text{proper binomial ideals of}\; \Tbq\;\}  \\
(L,\rho) &\longmapsto I(L,\rho).
\end{align*}
The inverse bijection sends a proper binomial ideal $J$ of $\Tbq$ to $(L,\rho)$ where
$$L = \{ \alpha \in \Gamma \mid x^\alpha \in \kx + J \},$$
and for each $\alpha \in L$, $\rho(\alpha)$ is the unique scalar in $\kx$ such that $x^\alpha - c(\alpha)^{-1} \rho(\alpha) \in J$.
\end{theorem}

\begin{proof}
By Theorem \ref{binom.Tq}, there are maps
$$
I : \calL \longleftrightarrow \{\;\text{proper binomial ideals of}\; \Tbq\;\} : \ell
$$
as described, and $I(\ell(J)) = J$ for all proper binomial ideals $J$ of $\Tbq$. 

It remains to show that $\ell(I(L,\rho)) = (L,\rho)$ for any $(L,\rho) \in \calL$.
Set $(L',\rho') := \ell(I(L,\rho))$, and observe that $L'$ is a sublattice of $\Gamma_Z$ containing $L$. If there exists $\alpha \in L' \setminus L$, there is a transversal for $L$ in $\Gamma$ that contains both $\alpha$ and $0$. But Proposition \ref{Irho.ideals}(c) would then imply that $x^\alpha$ and $1$ are linearly independent modulo $I(L,\rho)$, contradicting the fact that $x^\alpha \in \kx + I(L,\rho)$. Thus $L' = L$. That $\rho' = \rho$ then follows from the properness of $I(L,\rho)$.
Therefore $\ell(I(L,\rho)) = (L,\rho)$, as required.
\end{proof}

\subsection{}
Given an ideal $I$ of a general ring $R$, we use the commutative algebra notation $\sqrt I$ and the terminology \emph{radical of $I$} to stand for the prime radical of $I$, that is, the intersection of all prime ideals of $R$ that contain $I$. (By convention, the intersection of an empty family of ideals equals $R$, so that $\sqrt R = R$.) We also use the corresponding terminology \emph{radical ideal}, in place of ``semiprime ideal", for an ideal which equals its radical.

\begin{proposition}  \label{binomTq.radical}
If $\charr \KK = 0$, then all binomial ideals of $\Tbq$ are radical ideals.
\end{proposition}

\begin{proof} Let $I$ be a binomial ideal of $\Tbq$. Since the result is clear if $I=\Tbq$, we may assume that $I$ is proper. By Theorem \ref{binom.Tq} and Proposition \ref{Irho.ideals}, $\Tbq/I$ is a twisted group algebra of a group $\Gamma/L$ over $\KK$. Theorem 3.2 of \cite{MoPa} shows that $\Tbq/I$ is a semiprime ring, and therefore $I$ is a radical ideal of $\Tbq$.
\end{proof} 

Proposition \ref{binomTq.radical} typically fails in positive characteristic. For instance, if $n=1$ and $\charr \KK = p > 0$, the binomial ideal $\langle x_1^p - 1 \rangle = \langle x_1-1\rangle^p$ in $\Tbq$ is not radical.

\begin{theorem}  \label{maxTq}
{\rm(a)} For any $\rho \in \Hom(\Gamma_Z,\kx)$, the ideal $I(\Gamma_Z,\rho)$ is a maximal ideal of $\Tbq$.

{\rm(b)} If $\KK$ is algebraically closed, then every primitive ideal of $\Tbq$ is a maximal ideal of the form $I(\Gamma_Z,\rho)$ for some $\rho \in \Hom(\Gamma_Z,\kx)$.
\end{theorem}

\begin{proof}
(a) Observe from part (b) of Proposition \ref{Irho.ideals} that $\Zq/I_Z(\Gamma_Z,\rho)$ is $1$-dimensional over $\KK$. Hence, $I_Z(\Gamma_Z,\rho)$ is a maximal ideal of $\Zq$. Since also $I(\Gamma_Z,\rho) \cap \Zq = I_Z(\Gamma_Z,\rho)$, we conclude from Proposition \ref{idealsTq} that $I(\Gamma_Z,\rho)$ is a maximal ideal of $\Tbq$.

(b) Let $P$ be an arbitrary primitive ideal of $\Tbq$. It follows from \cite[Corollary II.8.5]{BG} that $P$ must be a maximal ideal. Specifically, with respect to the standard action of the torus $H = (\kx)^n$ on $\Tbq$, the zero ideal is the unique $H$-prime ideal of $\Tbq$, and so $\Spec \Tbq$ consists of a single $H$-stratum. Thus $P$, being maximal in its $H$-stratum, must be maximal in $\Spec \Tbq$.

Now by Proposition \ref{idealsTq}, $M := \phi^{-1}(P)$ is a maximal ideal of $\KK\Gamma_Z$. Since $\KK$ is algebraically closed, $M$ has codimension $1$ in $\KK\Gamma_Z$, and so $M = I_0(\Gamma_Z,\rho)$ for some $\rho \in \Hom(\Gamma_Z,\kx)$. Thus $P = \phi(M) \cdot \Tbq = I(\Gamma_Z,\rho)$ by Proposition \ref{idealsTq} and Theorem \ref{binom.Tq}(a).
\end{proof}

Recall that an ideal $P$ of a ring $R$ is called \emph{completely prime} if $R/P$ is a domain, i.e., a nonzero ring without nonzero zero-divisors.

\begin{theorem}  \label{prime.binomTq}
Let $L$ be a sublattice of $\Gamma_Z$ and  $\rho \in \Hom(L,\kx)$.

{\rm(a)} If $\Gamma_Z/L$ is torsionfree, then $I(L,\rho)$ is a prime ideal of $\Tbq$. If, further, $\Gamma/L$ is torsionfree, then $I(L,\rho)$ is completely prime.

{\rm(b)} If $\KK$ is algebraically closed, then $I(L,\rho)$ is a prime ideal of $\Tbq$ if and only if $\Gamma_Z/L$ is torsionfree. 

{\rm(c)} If $\KK$ is algebraically closed and $\Gamma/\Gamma_Z$ is torsionfree, then all prime binomial ideals of $\Tbq$ are completely prime.
\end{theorem}

\begin{proof}
(a) By Proposition \ref{Irho.ideals}, $\Zq/I_Z(L,\rho)$ is a twisted group algebra of $\Gamma_Z/L$ over $\KK$. Since $\Gamma_Z/L$ is torsionfree, it is a free abelian group of finite rank, and hence a totally ordered group under a suitable ordering. Consequently, $\Zq/I_Z(L,\rho)$ is a domain (e.g., \cite[Proposition 5.2.6]{NVO}), and $I_Z(L,\rho)$ is a prime ideal of $\Zq$. Since $I(L,\rho) \cap \Zq = I_Z(L,\rho)$, Proposition \ref{idealsTq} thus implies that $I(L,\rho)$ is a prime ideal of $\Tbq$.

Proposition \ref{Irho.ideals} also shows that $\Tbq/I(L,\rho)$ is a twisted group algebra of $\Gamma/L$ over $\KK$. In case $\Gamma/L$ is torsionfree, it is a free abelian group of finite rank, and consequently $\Tbq/I(L,\rho)$ is a domain. Therefore $I(L,\rho)$ is a completely prime ideal of $\Tbq$ in this case.

(b) One direction is given by part (a). Now assume that $I(L,\rho)$ is a prime ideal of $\Tbq$. By Propositions \ref{idealsTq} and \ref{Irho.ideals}, $I_Z(L,\rho)$ is a prime ideal of $Z(\Tbq)$ and $Z(\Tbq)/I_Z(L,\rho)$ is a twisted group algebra of $\Gamma_Z/L$ over $\KK$. In particular, this twisted group algebra is a domain. Since $\KK$ is algebraically closed, $\Gamma_Z/L$ must be torsionfree, as follows.

There is a $\KK$-basis $(b_\beta \mid \beta \in \Gamma_Z/L)$ for $Z(\Tbq)/I_Z(L,\rho)$ such that $b_\beta b_\gamma \in \kx b_{\beta+\gamma}$ for $\beta, \gamma \in \Gamma_Z/L)$. Suppose $\Gamma_Z/L$ contains a nonzero element $\delta$ with finite order $n$. Then $b_\delta^n = \lambda\cdot1$ for some $\lambda \in \kx$. Since $\KK$ is algebraically closed, there exists $\mu \in \kx$ with $\mu^n = \lambda$, and so
$$
(b_\delta - \mu\cdot1) ( b_\delta^{n-1} + \mu b_\delta^{n-2} +\cdots+ \mu^{n-1}\cdot1) = 0.
$$
This contradicts the fact that $Z(\Tbq)/I_Z(L,\rho)$ is a domain. Therefore $\Gamma_Z/L$ is torsionfree, as claimed.

(c) If $\Gamma/\Gamma_Z$ is torsionfree, then $\Gamma/L$ is torsionfree if and only if $\Gamma_Z/L$ is torsionfree. Therefore part (c) follows from (a) and (b), in view of Theorem \ref{binom.Tq}.
\end{proof}

\begin{remark}  \label{remarks.ILrho.pfime}
(1) Theorem \ref{prime.binomTq}(b) does not hold in general when $\KK$ is not algebraically closed. For instance, choose $\KK$ so that $\charr \KK \ne 2$ and $\KK$ contains an element $\alpha$ such that $X^4 - \alpha$ is irreducible in $\KK[X]$. Then consider $\Tbq$ with $n=2$ and $\bfq = \left[ \begin{smallmatrix} 1 &-1\\ -1 &1 \end{smallmatrix} \right]$. As is well known, $Z(\Tbq) = \KK[x_1^{\pm2}, x_2^{\pm2}]$ in this case, and so $\Gamma_Z = (2\ZZ)^2$. Choose $((2,0), (0,2))$ as basis for $\Gamma_Z$, define $\phi : \KK \Gamma_Z \rightarrow Z(\Tbq)$ as in \eqref{def.phi}, and observe that $\phi(z_{(4,0)}) = x_1^4$, whence $c(4,0) = 1$.

Now set $L := \ZZ(4,0) \subseteq \Gamma_Z$, and let $\rho \in \Hom(L,\kx)$ such that $\rho(4,0) = \alpha$. In view of Lemma \ref{ILrho.basis}, the ideal $I(L,\rho)$ of $\Tbq$ is generated by $x_1^4 - \alpha$. The quotient $\Tbq/I(L,\rho)$ is isomorphic to $K[x_2^{\pm1};\theta]$ where $K$ is the field $\KK[X]/(X^4-\alpha)$ and $\theta$ is the $\KK$-automorphism of $K$ sending the coset of $X$ to its negative. Thus, $\Tbq/I(L,\rho)$ is a domain and $I(L,\rho)$ is a prime ideal of $\Tbq$. However, the quotient $\Gamma_Z/L$ is not torsionfree.

(2) The hypothesis of torsionfreeness of $\Gamma/\Gamma_Z$ in Theorem \ref{prime.binomTq}(c) holds in particular in case the subgroup $\langle q_{ij} \rangle$ of $\kx$ is torsionfree. To see this, suppose $\alpha \in \Gamma$ and $t \in \Zpos$ with $t \alpha \in \Gamma_Z$. Then $x^{t\alpha} \in Z(\Tbq)$, and the relations $x_i x^{t \alpha} = x^{t \alpha} x_i$ imply that $\prod_{j=1}^n q_{ij}^{t \alpha_j} = 1$ for $i \in [1,n]$. If $\langle q_{ij} \rangle$ is torsionfree, it follows that $\prod_{j=1}^n q_{ij}^{\alpha_j} = 1$ and $x_i x^{\alpha} = x^{\alpha} x_i$ for $i \in [1,n]$, whence $x^\alpha \in Z(\Tbq)$ and $\alpha \in \Gamma_Z$.

Thus, if $\KK$ is algebraically closed and  $\langle q_{ij} \rangle$ is torsionfree, all prime binomial ideals of $\Tbq$ are completely prime. This, however, is covered by a more general theorem: Assuming only that $\langle q_{ij} \rangle$ is torsionfree,  all prime ideals of $\Abq$ are completely prime by \cite[Theorem 2.1]{GLet94}, whence all prime ideals of $\Tbq$ are completely prime.
\end{remark}

\begin{theorem}  \label{rad.binom}
Let $L$ be a sublattice of $\Gamma_Z$ and $\rho \in \Hom(L,\kx)$.

{\rm(a)} If $\charr\KK = 0$, then $\sqrt{I(L,\rho)} = I(L,\rho)$.

{\rm(b)} Assume that $\KK$ is perfect of characteristic $p > 0$.  Let $L'_p/L$ be the $p$-torsion subgroup of $\Gamma_Z/L$, that is,
$$
L'_p := \{ \alpha \in \Gamma_Z \mid p^l\alpha \in L\ \text{for some}\ l \in \Znn \}.
$$
There is a unique map $\rho'$ in $\Hom(L'_p, \kx)$ extending $\rho$, and $\sqrt{I(L,\rho)} = I(L'_p, \rho')$.
\end{theorem}

\begin{proof}
(a) Proposition \ref{binomTq.radical}.

(b) The existence and uniqueness of $\rho'$ follow from the fact that for any $l\ge0$, every element of $\kx$ has a unique $p^l$-th root in $\kx$. Note that $I(L,\rho) \subseteq I(L'_p,\rho')$.

Given any $\alpha \in L'_p$, there is some $l \ge 0$ such that $p^l \alpha \in L$, whence
$$(z_\alpha - \rho'(\alpha))^{p^l} = z_{p^l \alpha} - \rho(p^l \alpha) \in I_0(L,\rho).$$
Consequently,
$$(c(\alpha) x^\alpha - \rho'(\alpha))^{p^l} = \phi\bigl( (z_\alpha - \rho'(\alpha))^{p^l} \bigr) \in \phi(I_0(L,\rho)) \subseteq I(L,\rho),$$
and so $c(\alpha) x^\alpha - \rho'(\alpha) \in \sqrt{I(L,\rho)}$, due to $c(\alpha) x^\alpha - \rho'(\alpha)$ being central. Thus, $I(L'_p,\rho') \subseteq \sqrt{I(L,\rho)}$.

It remains to show that $I(L'_p,\rho')$ is a radical ideal. By construction, $\Gamma_Z/L'_p$ has no $p$-torsion, so the order of any finite subgroup of $\Gamma_Z/L'_p$ is nonzero in $\KK$. Since $\Zq/I_Z(L'_p,\rho')$ is a twisted group algebra of $\Gamma_Z/L'_p$ over $\KK$ by Proposition \ref{Irho.ideals}(b), \cite[Theorem I]{Pas2} shows that $\Zq/I_Z(L'_p,\rho')$ is a semiprime ring. Thus $I_Z(L'_p,\rho')$ is a radical ideal of $\Zq$.

Since $\sqrt{I(L'_p,\rho')}$ is the intersection of all prime ideals of $\Tbq$ containing $I(L'_p,\rho')$, Proposition \ref{idealsTq} implies that
$$\sqrt{I(L'_p,\rho')} \cap \Zq = \sqrt{I(L'_p,\rho') \cap \Zq} = \sqrt{I_Z(L'_p,\rho')} = I_Z(L'_p,\rho') = I(L'_p,\rho') \cap \Zq,$$
and therefore $\sqrt{I(L'_p,\rho')} = I(L'_p,\rho')$ as required.
\end{proof}

\begin{example}  \label{rad.nonbinom.Tq}
Perfectness of $\KK$ is needed in Theorem \ref{rad.binom} to obtain the map $\rho'$. If $\KK$ is not perfect, $\sqrt{I(L,\rho)}$ need not be a binomial ideal, as a slight modification of \cite[Lemma 29]{BGN} shows. Namely, let $p = \charr \KK >0$, take $n=2$ and $\bfq = \left[ \begin{smallmatrix} 1&1\\ 1&1 \end{smallmatrix} \right]$, and choose $a \in \KK \setminus \KK^p$. The polynomial $x_1^p - a \in \KK[x_1]$ is irreducible (e.g., \cite[Theorem 5.9.6]{Coh}), and so $\KK[x_1]/ \langle x_1^p - a \rangle$ is a field.

Take
$$I := \langle x_1^p - a,\, x_2^p - (a+1) \rangle = I(L,\rho)$$
where $L := (p\ZZ)^2$ while $\rho$ sends $(p,0) \mapsto a$ and $(0,p) \mapsto a+1$ (recall Lemma \ref{ILrho.basis}). Since
$$(x_1-x_2+1)^p = (x_1^p - a) - (x_2^p - a - 1) \in I$$
and $\Tbq/\langle x_1^p - a,\, x_1-x_2+1 \rangle \cong \KK[x_1]/\langle x_1^p - a \rangle$, we find that
$$\sqrt{I} = \langle x_1^p - a,\, x_1-x_2+1 \rangle$$
and $\dim \Tbq/\sqrt{I} = p$.

If $\sqrt{I} = I(L',\rho')$ for some sublattice $L'$ of $\Gamma$ and $\rho' \in \Hom(L',\kx)$, then Proposition \ref{Irho.ideals} shows that $\Tbq/\sqrt{I}$ is a twisted group algebra of $\Gamma/L'$ over $\KK$, with a $\KK$-basis of the form $(x^\sigma + \sqrt{I} \mid \sigma \in S )$ where $S$ is a transversal for $L'$ in $\Gamma$. Consequently, $\Tbq/\sqrt{I}$ is graded by $\Gamma/L'$, with $1$-dimensional homogeneous components $(\Tbq/\sqrt{I})_t = \KK (x^\sigma + \sqrt{I})$ for $t = \sigma+L' \in \Gamma/L'$.

The elements $y_i := x_i + \sqrt{I}$ of $\Tbq/\sqrt{I}$ are homogeneous. But $y_1 + \overline{1} = y_2$, so $y_1$ and $\overline 1$ must lie in the same homogeneous component, whence $y_1 \in \KK \overline 1$. Then $\dim \Tbq/\sqrt{I} = 1$, a contradiction.

Therefore $\sqrt{I}$ is not a binomial ideal of $\Tbq$.

In this example, $\sqrt{I}$ is a prime (actually, maximal) ideal of $\Tbq$ and hence is the unique prime ideal minimal over $I$. The example thus also shows that prime ideals minimal over a binomial ideal need not be binomial. Therefore the following theorem does not hold over arbitrary base fields.
\end{example}

\begin{theorem}  \label{minprimes.binom}
Assume that $\KK$ is algebraically closed. Let $L$ be a sublattice of $\Gamma_Z$ and $\rho \in \Hom(L,\kx)$. Let $L'/L$ be the torsion subgroup of $\Gamma_Z/L$, and let $\rho'_1,\dots,\rho'_m \in \Hom(L',\kx)$ be the distinct extensions of $\rho$ to $L'$.

The prime ideals of $\Tbq$ minimal over $I(L,\rho)$ are the same as the associated primes of $\Tbq/I(L,\rho)$, and they are the $I(L', \rho'_j)$ for $j \in [1,m]$.  All of these prime ideals have the same height, equal to $\rank L$.
\end{theorem}

\begin{proof} The commutative version of this result was proved in \cite{EiSt}. Following that paper but replacing $\ZZ^n$ by $\Gamma_Z$, let $p := \charr \KK \ge 0$, set $\Sat'_p(L) := L'$ in case $p = 0$, while if $p > 0$, let $\Sat'_p(L)$ be the largest sublattice of $L'$ such that $\Sat'_p(L) \supseteq L$ and $p \nmid |\Sat'_p(L)/L|$.

In terms of our present notation, \cite[Corollary 2.2]{EiSt} says the following: There are $m := |\Sat'_p(L)/L|$ distinct maps $\rho_1,\dots,\rho_m$ in $\Hom(\Sat'_p(L), \kx)$ extending $\rho$, and for $j \in [ 1,m ]$ there is a unique map $\rho'_j$ in $\Hom(L', \kx)$ extending $\rho_j$. The associated primes of $\KK\Gamma_Z/I_0(L,\rho)$ are $I_0(L', \rho'_j)$ for $j \in [ 1,m ]$, and these primes are all minimal over $I_0(L,\rho)$. Consequently, the $I_0(L', \rho'_j)$ for $j \in [ 1,m ]$ are exactly the prime ideals of $\KK\Gamma_Z$ minimal over $I_0(L,\rho)$. All of these prime ideals have the same height, equal to $\rank L$.

The theorem now follows from Proposition \ref{idealsTq}, Corollaries \ref{hts.Tq}, \ref{ass.Tq}, and Theorem \ref{binom.Tq}.
\end{proof}

\sectionnew{Binomial ideals in $\Abq$ and graded quotients}  \label{binomAq.1}
We initiate the study of binomials ideals in the quantum affine space $\Abq$ by developing gradings on quotients by such ideals. Quotients of quantum affine spaces by binomial ideals are characterized in terms of the existence of suitable gradings.  We also characterize twisted semigroup algebras of finitely generated commutative monoids and quantum affine toric varieties as quotients of quantum affine spaces modulo suitable binomial ideals.

\subsection{}
A \emph{binomial in $\Abq$} is any binomial of $\Tbq$ which lies in $\Abq$, i.e., any nonzero element $\lambda x^\alpha + \mu x^\beta$ where $\lambda, \mu \in \KK$ and $\alpha, \beta \in \Gamma^+$, and a \emph{binomial ideal of $\Abq$} is any ideal of $\Abq$ generated by binomials in $\Abq$. Similarly, a \emph{monomial in $\Abq$} is any nonzero element $\lambda x^\alpha$ where $\lambda \in \kx$ and $\alpha \in \Gamma^+$, and a \emph{monomial ideal of $\Abq$} is any ideal of $\Abq$ generated by monomials in $\Abq$.  Also, extending the usage in \cite{EiSt}, we define a (not necessarily commutative) \emph{binomial algebra} to be any algebra isomorphic to a quotient of a quantum affine space modulo a binomial ideal.

\begin{remark}
Although the algebra $\Tbq$ is best studied directly, one should note that it is also a binomial algebra.  Namely, given $\bfq$, define $\bfqhat = (\qhat_{ij}) \in M_{2n}(\kx)$ so that
\begin{equation*}
\begin{aligned}
\qhat_{ij} &:= q_{ij} &&(i,j \le n)  &\qquad\qhat_{ij} &:= q_{i-n,j-n} &&(i,j > n)  \\
\qhat_{ij} &:= q_{j-n,i} &&(i \le n < j) &\qhat_{ij} &:= q_{j,i-n} &&(i > n \ge j).
\end{aligned}
\end{equation*}
Then $\Tbq \cong \calO_{\bfqhat}(\KK^{2n}) / \langle x_i x_{i+n} - 1 \mid i \in [1,n] \rangle$.  Consequently, binomial ideals of $\Tbq$ correspond to quotients $I/ \langle x_i x_{i+n} - 1 \mid i \in [1,n] \rangle$ for certain binomial ideals $I$ of $\calO_{\bfqhat}(\KK^{2n})$.

Partial localizations of $\Abq$ obtained by inverting some of the $x_i$ may be described similarly.  For some other binomial algebras, see Theorems \ref{form.Aq.mod.binom}, \ref{qafftoric.as.Aq/B} and Remark \ref{more.binom.algs}.
\end{remark}

\subsection{}
Let $\le$ denote the lexicographic order on $\Gamma$, so that $(\Gamma,\le)$ is a totally ordered abelian group. Define \emph{degrees}, \emph{leading terms}, and \emph{leading coefficients} for nonzero elements of $\Tbq$ and $\Abq$ with respect to this ordering on exponents. Observe that $(\Gamma^+,\le)$ is a well-ordered set with least element $0$.

\begin{lemma}  \label{basis.binom.Aq}
Let $B$ be a binomial ideal of $\Abq$, and set
\begin{align*}
A_1 &:= \{ \alpha \in \Gamma^+ \mid x^\alpha \in B \}  \\
A_2 &:= \{ \alpha \in \Gamma^+ \setminus A_1 \mid x^\alpha - \mu x^{\alpha'} \in B\ \text{for some}\ \mu \in \kx,\; \alpha' \in \Gamma^+,\; \alpha' < \alpha \}.
\end{align*}
For $\alpha \in A_2$, let $p(\alpha)$ be the least element of $\Gamma^+$ such that $p(\alpha) < \alpha$ and $x^\alpha - \nu_\alpha x^{p(\alpha)} \in B$ for some $\nu_\alpha \in \kx$. Then $\nu_\alpha$ is uniquely determined by $\alpha$. 

Set $\calB_1 := (x^\alpha \mid \alpha \in A_1)$ and $\calB_2 := (x^\alpha - \nu_\alpha x^{p(\alpha)} \mid \alpha \in A_2)$. Then $\calB_1 \sqcup \calB_2$ is a $\KK$-basis for $B$.
\end{lemma}

\begin{proof}
If $B=0$, then $\calB_1$ and $\calB_2$ are empty and the result is clear. Hence, we may assume that $B \ne 0$. For $\alpha \in A_2$, note that since $x^\alpha - \nu_\alpha x^{p(\alpha)} \in B$ while $x^\alpha \notin B$, we must have $x^{p(\alpha)} \notin B$ and $p(\alpha) \notin A_1$. It follows that $\nu_\alpha$ is uniquely determined by $\alpha$.

If $\calB_1 \sqcup \calB_2$ fails to be linearly independent, there is a nontrivial relation
\begin{equation}  \label{lin.dep.B1B2}
\sum_{i=1}^m \lambda_i x^{\alpha_i} + \sum_{j=1}^r \mu_j \bigl( x^{\gamma_j} - \nu_{\gamma_j} x^{p(\gamma_j)} \bigr) = 0
\end{equation}
for some distinct $\alpha_1,\dots,\alpha_m \in A_1$, some distinct $\gamma_1,\dots,\gamma_r \in A_2$, and some $\lambda_i, \mu_j \in \kx$, where $m+r>0$. Since $\gamma_j, p(\gamma_j) \notin A_1$ for all $j$, there cannot be any $\lambda_i x^{\alpha_i}$ terms in \eqref{lin.dep.B1B2}, i.e., $m=0$ and
\begin{equation}  \label{lin.dep.B2}
\sum_{j=1}^r \mu_j \bigl( x^{\gamma_j} - \nu_{\gamma_j} x^{p(\gamma_j)} \bigr) = 0.
\end{equation}
Without loss of generality, $\gamma_1 < \cdots < \gamma_r$, so also $p(\gamma_j) < \gamma_j \le \gamma_r$ for all $j$. But then the left hand side of \eqref{lin.dep.B2} is nonzero with leading term $\mu_r x^{\gamma_r}$, which is impossible. Therefore $\calB_1 \sqcup \calB_2$ is $\KK$-linearly independent.

By hypothesis, $B$ is generated by certain binomials $\lambda_i x^{\alpha_i} + \mu_i x^{\beta_i}$, whence elements of $B$ are linear combinations of products of the form $x^\gamma ( \lambda_i x^{\alpha_i} + \mu_i x^{\beta_i} ) x^\delta$ with $\gamma,\delta \in \Gamma^+$. Any such product is itself a binomial, and so $B$ is spanned by binomials. Thus, to prove that $\calB_1 \sqcup \calB_2$ spans $B$, it suffices to show that every binomial $\lambda x^\alpha + \mu x^\beta$ in $B$ is in the $\KK$-span of $\calB_1 \sqcup \calB_2$.

If $\mu = 0$, then $\lambda \ne 0$ and $x^\alpha \in B$, whence $x^\alpha \in \calB_1$ and $\lambda x^\alpha + \mu x^\beta \in \Span(\calB_1)$. The same conclusion holds if $\lambda = 0$, or if $\alpha = \beta$. Consequently, we may assume that $\lambda,\mu \ne 0$ and $\alpha \ne \beta$. After switching terms if necessary, we may also assume that $\alpha > \beta$. We now proceed by induction on $\alpha$.

Since $x^\alpha + \lambda^{-1} \mu x^\beta \in B$, we must have $\alpha \in A_2$. Hence, the first step of the induction is the case when $\alpha = \min(A_2)$. Now $B$ contains the element
\begin{equation}  \label{B.elt}
( \lambda x^\alpha + \mu x^\beta ) - \lambda ( x^\alpha - \nu_\alpha x^{p(\alpha)} ) = \mu x^\beta  + \lambda \nu_\alpha x^{p(\alpha)}.
\end{equation}
If $p(\alpha) \ne \beta$, we find that $\max(p(\alpha), \beta) \in A_2$, which is impossible since $p(\alpha)$ and $\beta$ are both $< \alpha$. Hence, $p(\alpha) = \beta$ and the right hand side of \eqref{B.elt} is zero or a monomial. It follows that this right hand term lies in $\Span(\calB_1)$, whence $\lambda x^\alpha + \mu x^\beta \in \Span(\calB_1 \sqcup \calB_2)$.

Finally, suppose that $\alpha > \min(A_2)$. As above, $B$ contains the element \eqref{B.elt}, and if $p(\alpha) = \beta$, then $\lambda x^\alpha + \mu x^\beta \in \Span(\calB_1 \sqcup \calB_2)$. If $p(\alpha) \ne \beta$, then either $p(\alpha) < \beta < \alpha$ or $\beta < p(\alpha) < \alpha$. In either case, our induction hypothesis implies that $\mu x^\beta  + \lambda \nu_\alpha x^{p(\alpha)} \in \Span(\calB_1 \sqcup \calB_2)$, whence also $\lambda x^\alpha + \mu x^\beta \in \Span(\calB_1 \sqcup \calB_2)$. This completes the induction and concludes the proof that $\calB_1 \sqcup \calB_2$ spans $B$.
\end{proof} 

Analogous to the case of twisted group algebras, a $\KK$-algebra $R$ is a twisted semigroup algebra of an additive monoid $S$ over $\KK$, denoted $\KK^eS$, if and only if $R$ has a $\KK$-basis $(z_s \mid s \in S)$ such that $z_s z_t \in \kx z_{s+t}$ for all $s,t \in S$. (One has only to check that the map $e : S \times S \rightarrow \kx$ such that $z_s z_t = e(s,t) z_{s+t}$ is a $2$-cocycle on $S$.)

\begin{lemma}  \label{Aq.mod.binom}
Let $B$ be a binomial ideal of $\Abq$, and set $y^\alpha := x^\alpha + B \in \Abq/B$ for $\alpha \in \Gamma^+$. There is a monoid congruence $\sim$ on $\Gamma^+$ defined by the rule $\alpha \sim \beta \iff \KK y^\alpha = \KK y^\beta$. Let $S := \Gamma^+/{\sim}$, and for $s\in S$ set $R_s := \KK y^\alpha$ where $\alpha$ is some {\rm(}any{\rm)} element of $s$. Then
$$\Abq/B = \bigoplus_{s\in S} R_s$$
is a grading of $\Abq/B$ by the monoid $S$.

If $B$ contains no monomials, i.e., $x^\alpha \notin B$ for all $\alpha \in \Gamma^+$, then $\Abq/B$ is isomorphic to a twisted semigroup algebra $\KK^eS$.
\end{lemma}

\begin{proof} It is clear that $\sim$ is a monoid congruence on $\Gamma^+$. Set $R := \Abq/B$, and observe that
\begin{equation}  \label{sum.Rs}
R = \sum_{\alpha \in \Gamma^+} \KK y^\alpha = \sum_{s\in S} R_s \,.
\end{equation}
Given any $s,t \in S$, choose $\alpha \in s$ and $\beta \in t$, and note that
\begin{equation}  \label{RsRt}
R_s R_t = \KK y^\alpha y^\beta = \KK y^{\alpha+\beta} = R_{s+t} \,.
\end{equation}
Thus, all that remains is to prove that the right hand sum in \eqref{sum.Rs} is direct.

We will use the notation of Lemma \ref{basis.binom.Aq} and its proof. Note that for $\alpha \in A_2$, we have $y^\alpha = \nu_\alpha y^{p(\alpha)}$, whence $\alpha \sim p(\alpha)$. For $s\in S$, set $\alpha_s := \min(s)$. Then $R_s = \KK y^{\alpha_s}$, and $R_s = 0$ if and only if $\alpha_s \in A_1$. To prove that the sum $\sum_{s\in S} R_s$ is direct, it thus suffices to show that the $y^{\alpha_s}$, for $s\in S$ such that $\alpha_s \notin A_1$, are $\KK$-linearly independent.

Suppose there is a nontrivial relation $\sum_{i=1}^m \lambda_i y^{\alpha_{s(i)}} = 0$ for some $\lambda_i \in \kx$ and some distinct $s(i) \in S$ with $\alpha_{s(i)} \notin A_1$. We may assume that $\alpha_{s(1)} < \cdots < \alpha_{s(m)}$. In view of Lemma \ref{basis.binom.Aq}, 
\begin{equation}  \label{lin.dep.modB}
\sum_{i=1}^m \lambda_i x^{\alpha_{s(i)}} = \sum_{j=1}^q \mu_j x^{\beta_j} + \sum_{k=1}^r \mu'_k \bigl( x^{\gamma_k} - \nu_{\gamma_k} x^{p(\gamma_k)} \bigr)
\end{equation}
for some $\mu_j, \mu'_k \in \kx$, some distinct $\beta_j \in A_1$, and some distinct $\gamma_k \in A_2$. Since none of the $\alpha_{s(i)}$, $\gamma_k$, or $p(\gamma_k)$ are in $A_1$, there cannot be any $\mu_j x^{\beta_j}$ terms in \eqref{lin.dep.modB}, so this equation reduces to
\begin{equation}  \label{l.d.modB'}
\sum_{i=1}^m \lambda_i x^{\alpha_{s(i)}} = \sum_{k=1}^r \mu'_k \bigl( x^{\gamma_k} - \nu_{\gamma_k} x^{p(\gamma_k)} \bigr).
\end{equation}
We may assume that $\gamma_1 < \cdots < \gamma_r$. Comparing the degrees of the left and right sides of \eqref{l.d.modB'}, we see that $\alpha_{s(m)} = \gamma_r$. This means that $\gamma_r$ is the least element of the congruence class $s(m)$. Since $p(\gamma_r) \sim \gamma_r$ and $p(\gamma_r) < \gamma_r$, this is impossible.
Therefore $( y^{\alpha_s} \mid s\in S,\; \alpha_s \notin A_1 )$ is $\KK$-linearly independent, as desired.  This completes the proof of the first conclusion of the lemma.

Finally, assume that $x^\alpha \notin B$ for all $\alpha \in \Gamma^+$, so that $y^\alpha \ne 0$ in $R$ for all $\alpha$.  For $s \in S$, choose and fix $\alpha(s) \in s$, and set $z_s := y^{\alpha(s)}$.  Then $(z_s \mid s \in S)$ is a $\KK$-basis for $R$.  Since $y^\alpha y^\beta = d(\alpha,\beta) y^{\alpha+\beta} \ne 0$ for all $\alpha, \beta \in \Gamma^+$, we have $z_s z_t \in \kx z_{s+t}$ for all $s,t \in S$.  Therefore $R$ is a twisted semigroup algebra of $S$ over $\KK$.
\end{proof} 

\begin{lemma}  \label{kerphi.binom}
Let $R$ be a unital affine $\KK$-algebra which has a grading by a commutative monoid $S$, and let $\{r_1,\dots,r_n\}$ be a set of homogeneous $\KK$-algebra generators of $R$. Assume that the homogeneous components $R_s$ of $R$ have $\KK$-dimension $\le 1$, and that there is a multiplicatively skew-symmetric matrix $\bfq := (q_{ij}) \in M_n(\kx)$ such that $r_i r_j = q_{ij} r_j r_i$ for all $i,j \in [ 1,n ]$. Then there is a unique surjective $\KK$-algebra map $\psi : \Abq \rightarrow R$ such that $\psi(x_i) = r_i$ for all $i$, and $\ker \psi$ is a binomial ideal of $\Abq$.
\end{lemma}

\begin{proof}
The existence and uniqueness of $\psi$ are clear.  Keeping $r_1,\dots,r_n$ fixed, set $r^\alpha := r_1^{\alpha_1} \cdots r_n^{\alpha_n}$ for all $\alpha \in \Gamma^+$; then $\psi(x^\alpha) = r^\alpha$. Also, set $B := \ker\psi$. For $i \in [1,n]$, let $s_i := \deg(r_i)$ with respect to the $S$-grading on $R$, so that $r_i \in R_{s_i}$.

Let $B'$ be the ideal of $\Abq$ generated by all those binomials of $\Abq$ that lie in $B$. Set $R' := \Abq/B'$ and write $y^\alpha := x^\alpha + B' \in R'$ for $\alpha \in \Gamma^+$. By Lemma \ref{Aq.mod.binom}, there is a monoid congruence $\sim$ on $\Gamma^+$ defined by the rule $\alpha \sim \beta \iff \KK y^\alpha = \KK y^\beta$, and $R'$ is graded by the monoid $S' := \Gamma^+/{\sim}$ where $R'_{s'} := \KK y^\alpha$ for $s' \in S'$ and $\alpha \in s'$. Note that $\psi$ induces a surjective $\KK$-algebra map $\psibar : R' \rightarrow R$ such that $\psibar(y^\alpha) = r^\alpha$ for $\alpha \in \Gamma^+$. We claim that $\psibar$ is an isomorphism. Since $\ker \psibar = B/B'$, we will then obtain $B = B'$, proving that $B$ is a binomial ideal of $\Abq$.

We have $R' = \bigoplus_{s' \in S'_\bullet} R'_{s'}$ where $S'_\bullet := \{ s' \in S' \mid R'_{s'} \ne 0 \}$. For $s' \in S'_\bullet$, choose and fix $\alpha(s') \in s'$, so that $R'_{s'} = \KK y^{\alpha(s')}$. Then $( y^{\alpha(s')} \mid s' \in S'_\bullet )$ is a $\KK$-basis for $R'$. For $s' \in S'_\bullet$, we have $x^{\alpha(s')} \notin B'$ and so $x^{\alpha(s')} \notin B$, whence $r^{\alpha(s')} \ne 0$. There is a monoid homomorphism $\pi : \Gamma^+ \rightarrow S$ such that $\pi(\alpha) = \alpha_1 s_1 + \cdots + \alpha_n s_n$ for $\alpha \in \Gamma^+$, and $r^\alpha \in R_{\pi(\alpha)}$. Since $\dim R_s \le 1$ for $s\in S$, it follows that $R_{\pi\alpha(s')} = \KK r^{\alpha(s')}$ for $s' \in S'_\bullet$.

We claim that the map $S'_\bullet \rightarrow S$ given by $s' \mapsto \pi\alpha(s')$ is injective. Suppose $s',t' \in S'_\bullet$ with $\pi\alpha(s') = \pi\alpha(t')$. Then $R_{\pi\alpha(s')} = R_{\pi\alpha(t')}$ and so $r^{\alpha(s')} = \lambda r^{\alpha(t')}$ for some $\lambda \in \kx$. Consequently, the binomial $x^{\alpha(s')} - \lambda x^{\alpha(t')}$ lies in $B$ and hence also in $B'$. It follows that $y^{\alpha(s')} = \lambda y^{\alpha(t')}$ and so $\alpha(s') \sim \alpha(t')$, which, finally, implies $s' = t'$.

Since the  $r^{\alpha(s')}$ for $s' \in S'_\bullet$ are nonzero and lie in distinct homogeneous components $R_{\pi\alpha(s')}$, they are $\KK$-linearly independent. Moreover, $\psibar$ maps the $\KK$-basis $( y^{\alpha(s')} \mid s' \in S'_\bullet )$ of $R'$ to the family $( r^{\alpha(s')} \mid s' \in S'_\bullet )$. Therefore $\psibar$ is injective and thus an isomorphism, as claimed.
\end{proof}

The grading information in Lemma \ref{kerphi.binom} allows us to show that binomial ideals of $\Tbq$ contract to binomial ideals of $\Abq$.

\begin{proposition}  \label{contract.binom}
If $I$ is a binomial ideal of $\Tbq$, then $I \cap \Abq$ is a binomial ideal of $\Abq$.
\end{proposition}

\begin{proof} Since the result is clear if $I = \Tbq$, we may assume that $I$ is a proper ideal of $\Tbq$. By Theorem \ref{binom.Tq} and Proposition \ref{Irho.ideals}(c), $I = I(L,\rho)$ for some sublattice $L$ of $\Gamma_Z$ and $\rho \in \Hom(L,\kx)$, and $\Tbq/I$ is a twisted group algebra of $\Gamma/L$ over $\KK$ with a $\KK$-basis of the form $( x^\tau + I \mid \tau \in T )$ where $T$ is any transversal for $L$ in $\Gamma$. In particular, $\Tbq/I$ is graded by the group $\Gamma/L$, with $1$-dimensional homogeneous components $(\Tbq/I)_{\tau+L} = \KK (x^\tau + I)$ for $\tau \in T$.

Now $\Abq/(I \cap \Abq)$ is isomorphic to the subalgebra $R$ of $\Tbq/I$ generated by the cosets $r_i := x_i + I$ for $i \in [ 1,n ]$. Note that the $r_i$ are homogeneous of degree $e_i + L$, where $(e_1,\dots,e_n)$ is the standard basis for $\Gamma$. Since the $r_i$ are homogeneous, $R$ is graded by the group $S := \Gamma/L$, which we also view as a commutative monoid. Note that the nonzero homogeneous components of $R$ are $1$-dimensional over $\KK$.

The composition of the inclusion map $\Abq \rightarrow \Tbq$ with the quotient map $\Tbq \rightarrow \Tbq/I$ provides a surjective $\KK$-algebra map $\psi : \Abq \rightarrow R$ such that $\psi(x_i) = r_i$ for all $i$ and $\ker \psi = I \cap \Abq$. Therefore Lemma \ref{kerphi.binom} implies that $I\cap \Abq$ is a binomial ideal of $\Abq$.
\end{proof}

\begin{theorem}  \label{form.Aq.mod.binom}
Let $R$ be a unital affine $\KK$-algebra. The following are equivalent:

{\rm(a)} $R$ is a binomial $\KK$-algebra.

{\rm(b)} $R$ has a grading by a commutative monoid $S$  such that \begin{enumerate}
\item all nonzero homogeneous components $R_s$ of $R$ are $1$-dimensional over $\KK$;
\item for any $s,t \in S$, we have $R_s R_t = 0$ if and only if $R_t R_s = 0$.
\end{enumerate}
\end{theorem}

\begin{proof}
(a)$\Longrightarrow$(b): Without loss of generality, $R = \Abq/B$ for some multiplicatively skew-symmetric matrix $\bfq \in M_n(\kx)$, some $n \in \Zpos$, and some binomial ideal $B$ of $\Abq$. By Lemma \ref{Aq.mod.binom}, there is a grading of $R$ by a monoid $S = \Gamma^+/{\sim}$ with $R_s = \KK y^\alpha$ for $s\in S$ and $\alpha \in s$, hence Condition (1) holds.  Given $s,t \in S$, choose $\alpha \in s$ and $\beta \in t$. As noted in \eqref{RsRt}, $R_sR_t = R_{s+t}$, whence also $R_t R_s = R_{s+t}$. Condition (2) follows.

(b)$\Longrightarrow$(a): We may assume that $R \ne 0$. Since $R$ is affine, it can be generated by finitely many nonzero homogeneous elements, say $r_1,\dots,r_n$ where $r_i \in R_{s_i}$ for some $s_i \in S$.

Assumption (b)(1) implies that $R_{s_i} = \KK r_i$ for $i \in [ 1,n ]$, and then assumption (b)(2) implies that $r_ir_j = 0 \iff r_jr_i = 0$ for $i,j \in [ 1,n ]$. Given $i,j \in [ 1,n ]$ with $r_ir_j \ne 0$, we then have $R_{s_i+s_j} = \KK r_ir_j = \KK r_jr_i$, so there exist nonzero scalars $q_{ij}$ and $q_{ji} = q_{ij}^{-1}$ in $\kx$ such that
\begin{equation}  \label{build.qij}
r_i r_j = q_{ij} r_j r_i \qquad\text{and}\qquad r_j r_i = q_{ji} r_i r_j \,.
\end{equation}
Note that if $i=j$, we must have $q_{ij} = 1$.
On the other hand, if $r_ir_j = r_jr_i = 0$, we can write \eqref{build.qij} with $q_{ij} := 1 =: q_{ji}$. We now have a multiplicatively skew-symmetric matrix $\bfq := (q_{ij}) \in M_n(\kx)$ such that $r_i r_j = q_{ij} r_j r_i$ for all $i,j \in [ 1,n ]$.

There is a unique surjective $\KK$-algebra map $\psi : \Abq \rightarrow R$ such that $\psi(x_i) = r_i$ for all $i$, and Lemma \ref{kerphi.binom} shows that $B := \ker \psi$ is a binomial ideal of $\Abq$. This establishes (a).
\end{proof}

\begin{corollary}  \label{KcS=Aq/B}
Let $R$ be a unital $\KK$-algebra. Then $R$ is isomorphic to a twisted semigroup algebra $\KK^eS$ for some finitely generated commutative monoid $S$ if and only if $R \cong \Abq/B$ for some multiplicatively skew-symmetric matrix $\bfq \in M_n(\kx)$, some $n \in \Zpos$, and some binomial ideal $B$ of $\Abq$ that contains no monomials $x^\alpha$.
\end{corollary}

\begin{proof}
Sufficiency is given by Lemma \ref{Aq.mod.binom}.

Conversely, assume that $R \cong \KK^eS$ for some finitely generated commutative monoid $S$ and some $2$-cocycle $e : S\times S \rightarrow \kx$. Then $R$ has a $\KK$-basis $(r_s \mid s\in S)$ such that $r_s r_t = e(s,t) r_{s+t}$ for all $s,t \in S$.  Setting $R_s := \KK r_s$ for $s \in S$, we obtain an $S$-grading on $R$ such that conditions (b)(1),(2) of Theorem \ref{form.Aq.mod.binom} hold.  By that theorem, there exists a surjective $\KK$-algebra map $\phi : \Abq \rightarrow R$ where $\bfq \in M_n(\kx)$ is multiplicatively skew-symmetric, $n \in \Zpos$, and $B := \ker \phi$ is a binomial ideal of $\Abq$.  The proof of Theorem \ref{form.Aq.mod.binom} shows that $\phi$ may be chosen so that each $\phi(x_i) = r_{s_i}$ for some $s_i \in S$.  Consequently, for each $\alpha \in \Gamma^+$, there is some $s \in S$ such that $\phi(\KK x^\alpha) = \KK r_s \ne 0$, and therefore $x^\alpha \notin B$.
\end{proof}

\begin{example}
Condition (b)(2) in Theorem \ref{form.Aq.mod.binom} cannot be omitted, as the following example shows:
$$R := \KK \langle X,Y \rangle / \langle X^3, Y^2, XY, YX^2 \rangle.$$
If $x$ and $y$ denote the cosets of $X$ and $Y$ in $R$, then $(1,x,y,x^2,yx)$ is a $\KK$-basis for $R$. Clearly $R$ can be graded by the monoid $\Znn^2$, with $x \in R_{(1,0)}$ and $y \in R_{(0,1)}$. All nonzero homogeneous components of this grading are $1$-dimensional.

As is easily checked, $Z(R) = \KK + \KK x^2 + \KK yx$, and the group of units of $R$ is the set $\kx + \KK x + \KK y + \KK x^2 + \KK yx$. Further, the only normal elements in $R$ are the central elements together with the units (we leave this calculation to the reader). 

We claim that if $u$ and $v$ are any normal elements in $R$ with $uv = qvu \ne 0$ for some $q \in \kx$, then $q=1$. This is clear if $u$ or $v$ is central. Otherwise, $u$ and $v$ are units which commute modulo the ideal $\KK x + \KK y + \KK x^2 + \KK yx$, and again $q=1$. Thus, any pair of quasi-commuting normal elements of $R$ actually commute. Since $R$ is not commutative, it therefore cannot be a quotient of any quantum affine space.
\end{example}

We recall that a \emph{quantum affine toric variety} over $\KK$ (cf.~\cite{Ing, RZ}) is a twisted semigroup algebra $A = \KK^eS$ where $S$ is a finitely generated submonoid of a free abelian group of finite rank, i.e., $A$ is a subalgebra of a quantum torus $\Tbq$ of the form $\sum_{\alpha\in S} \KK x^\alpha$  for some finitely generated submonoid $S$ of $\Gamma$. Equivalently \cite{Ing}, a quantum affine toric variety is an affine domain $A$ over $\KK$ equipped with a rational action of an algebraic torus $H = (\kx)^r$ by $\KK$-algebra automorphisms such that the $H$-eigenspaces of $A$ are $1$-dimensional.

\begin{theorem}  \label{qafftoric.as.Aq/B}
Let $R$ be a unital $\KK$-algebra. If $R$ is isomorphic to a quantum affine toric variety over $\KK$, then $R \cong \Abq/B$ for some multiplicatively skew-symmetric matrix $\bfq \in M_n(\kx)$, some $n \in \Zpos$, and some completely prime binomial ideal $B$ of $\Abq$. The converse holds if $\KK$ is algebraically closed.
\end{theorem}

\begin{proof}
The first implication is immediate from Theorem \ref{form.Aq.mod.binom}.

Now assume that $\KK$ is algebraically closed, and that $R \cong \Abq/B$ for some $\Abq$ and some completely prime binomial ideal $B$ of $\Abq$.  Set $J := \{ j \in [1,n] \mid x_j \in B \}$ and $I := \langle x_j \mid j \in J \rangle \vartriangleleft \Abq$.  Then $R \cong (\Abq/I)/(B/I)$.  Since $\Abq/I$ is a quantum affine space, we may replace $\Abq$ and $B$ by $\Abq/I$ and $B/I$, and so there is no loss of generality in assuming that $x_i \notin B$ for all $i$.  As $B$ is completely prime, it follows that $x^\alpha \notin B$ for all $\alpha \in \Gamma^+$. 

By Corollary \ref{KcS=Aq/B}, $R \cong \KK^eS$ for some finitely generated commutative monoid $S$ and some $2$-cocycle $e : S\times S \rightarrow \kx$.  Hence, $R$ has a $\KK$-basis $(z_s \mid s \in S)$ such that $z_s z_t = e(s,t) z_{s+t}$ for all $s,t \in S$.

We next claim that $S$ is cancellative. Suppose $s,t,u \in S$ with $s+u = t+u$.  Then $z_s z_u = e(s,u) z_{s+u} = e(s,u) e(t,u)^{-1} z_t z_u$.  Since $R$ is a domain, $z_s = e(s,u) e(t,u)^{-1} z_t$, whence $s = t$ as desired.

Let $G$ be the universal group of $S$, an abelian group containing $S$ as a submonoid which generates $G$. Thus, $G$ is finitely generated. We claim that $G$ is torsionfree. If not, it contains a nonzero element $g$ with finite order $m>1$. Write $g = s-t$ for some distinct $s,t \in S$.  Since $z_s^m \in \kx z_{ms}$ and $z_t^m \in \kx z_{mt}$, there is some $\lambda \in \kx$ such that $z_s^m = \lambda z_t^m$.  We also have $z_t z_s = \eta z_s z_t$ for some $\eta \in \kx$.  As $\KK$ is algebraically closed, $\lambda \eta^{- m(m-1)/2} = \mu^m$ for some $\mu \in \kx$. Set
$$
w := \sum_{i=1}^m \mu^{i-1} \eta^{(i-1)m - i(i-1)/2} z_s^{m-i} z_t^{i-1} \,,
$$
and calculate that $(z_s - \mu z_t) w = z_s^m - \lambda z_t^m = 0$.  Since 
$$
z_s^{m-i} z_t^{i-1} \in \kx z_{(m-i)s+(i-1)t}
$$
and $(m-i)s+(i-1)t \ne (m-j)s+(j-1)t$ for $1 \le i < j \le m$, the elements $z_s^{m-i} z_t^{i-1}$ for $i \in [1,m]$ are $\KK$-linearly independent, whence $w \ne 0$. Consequently, $z_s - \mu z_t = 0$. However, this forces $s = t$, contradicting the assumption that $g \ne 0$. Thus $G$ is torsionfree, as claimed.

Now $G$ is free abelian of finite rank, and therefore $R$ is a quantum affine toric variety over $\KK$.
\end{proof}

The converse implication in Theorem \ref{qafftoric.as.Aq/B} can easily fail when $\KK$ is not algebraically closed, e.g., in the case when $\Abq = \RR[x]$ and $B = \langle x^2+1 \rangle$.

\begin{remark}  \label{more.binom.algs}
Eisenbud and Sturmfels showed in \cite[Corollary 1.9]{EiSt} that various algebras associated with a monomial ideal in a binomial algebra are themselves binomial algebras.  There are good noncommutative analogs for blowup, Rees, and associated graded algebras, as follows.

Let $R$ be a unital affine binomial $\KK$-algebra, graded by a commutative monoid $S$ as in Theorem \ref{form.Aq.mod.binom}(b), and let $\sigma$ be a $\KK$-algebra automorphism of $R$ such that the homogeneous components $R_s$ of $R$ are all $\sigma$-invariant.  The corresponding \emph{skew-Laurent algebra}, $R[z^{\pm1};\sigma]$, is a free left $R$-module with basis $(z^n \mid n \in \ZZ)$, and with $zr = \sigma(r)z$ for $r \in R$.  The $S$-grading on $R$ extends naturally to an $(S\times\ZZ)$-grading on $R[z^{\pm1};\sigma]$, where $R[z^{\pm1};\sigma]_{(s,m)} := R_s z^m$ for $(s,m) \in S\times\ZZ$.  Theorem \ref{form.Aq.mod.binom} implies that $R[z^{\pm1};\sigma]$ is a binomial $\KK$-algebra, as is any unital affine $(S\times\ZZ)$-homogeneous subalgebra.  

Now suppose that $I$ is a finitely generated $(S\times\ZZ)$-homogeneous $\sigma$-invariant ideal of $R$.  Then, by the above, the algebras
$$
\calB(R,I,\sigma) := \bigoplus_{m=0}^\infty I^m z^m \qquad\text{and}\qquad \calR(R,I,\sigma) := \bigoplus_{m=-\infty}^{-1} R z^m \oplus \bigoplus_{m=0}^\infty I^m z^m
$$
are binomial $\KK$-algebras.  The associated graded algebra $\gr_I R = \bigoplus_{m=0}^\infty I^m/I^{m+1}$ is isomorphic to $\calB(R,I,\id)/I \calB(R,I,\id)$, whence it is graded by $S\times\ZZ$ with homogeneous components having dimension at most $1$.  The conditions on $R$ imply that
$$
R[z^{\pm1};\id]_{(s,m)} R[z^{\pm1};\id]_{(t,l)} = R[z^{\pm1};\id]_{(s+t,m+l)} = R[z^{\pm1};\id]_{(t,l)}R[z^{\pm1};\id]_{(s,m)}
$$
for all $(s,m),(t,l) \in S\times\ZZ$, from which it follows that the analogous property holds in $\calB(R,I,\id)$ and in $\gr_I R$.  Therefore another application of Theorem \ref{form.Aq.mod.binom} implies that $\gr_I R$ is a binomial $\KK$-algebra.
\end{remark}

\sectionnew{Binomial ideals of $\Abq$ not containing monomials}  \label{binom.Aq.w/o.monom}

Those binomial ideals of $\Abq$ that contain no monomials in the generators $x_i$ are closely related to binomial ideals of $\Tbq$ via localization and contraction. In particular, prime (binomial) ideals of $\Abq$ not containing monomials correspond precisely to prime (binomial) ideals of $\Tbq$. This allows the results of Section \ref{binomTq} on radicals, minimal and associated primes to be carried over to binomial ideals of $\Abq$ not containing monomials. We deal with general binomial ideals of $\Abq$ in Section \ref{genbinom.Aq}.

\subsection{}
For $\alpha \in \Gamma$, let $\alpha_{\pm}$ denote the \emph{positive} and \emph{negative parts} of $\alpha$:
\begin{align*}
\alpha_+ &:= (\max(\alpha_1,0), \dots, \max(\alpha_n,0))  \\
\alpha_- &:= (- \min(\alpha_1,0), \dots, - \min(\alpha_n,0)),
\end{align*} 
so that $\alpha_{\pm} \in \Gamma^+$ and $\alpha = \alpha_+ - \alpha_-$. Note that $(\alpha_+)_i$ and $(\alpha_-)_i$ cannot both be positive for any $i \in [ 1,n ]$, whence $\sum_{i=1}^n (\alpha_+)_i (\alpha_-)_i = 0$. Let us abbreviate this condition as $\alpha_+ \perp \alpha_-$. Conversely, if $\beta, \gamma \in \Gamma^+$ with $\beta \perp \gamma$, then $\beta = (\beta - \gamma)_+$ and $\gamma = (\beta - \gamma)_-$.

For any sublattice $L$ of $\Gamma_Z$ and $\rho \in \Hom(L,\kx)$, define
$$I_A(L,\rho) := \langle x^{\alpha_+} - c(\alpha)^{-1} d(\alpha, \alpha_-)^{-1} \rho(\alpha) x^{\alpha_-} \mid \alpha \in L \rangle \vartriangleleft \Abq \,.$$

\begin{theorem}  \label{IcapAq}
If $L$ is a sublattice of $\Gamma_Z$ and $\rho \in \Hom(L,\kx)$, then
$$I(L,\rho) \cap \Abq = I_A(L,\rho) .$$
\end{theorem}

\begin{proof}
For $\alpha \in L$, observe that $I(L,\rho)$ contains the element
$$\bigl[ x^\alpha - c(\alpha)^{-1} \rho(\alpha) \bigr] x^{\alpha_-} = d(\alpha, \alpha_-) \bigl[ x^{\alpha_+} - c(\alpha)^{-1} d(\alpha, \alpha_-)^{-1} \rho(\alpha) x^{\alpha_-} \bigr].$$
Thus, $I_A(L,\rho) \subseteq I(L,\rho) \cap \Abq$.

To establish the reverse inclusion, it suffices to show that all binomials in $I(L,\rho) \cap \Abq$ lie in $I_A(L,\rho)$, since $I(L,\rho) \cap \Abq$ is a binomial ideal of $\Abq$ (Proposition \ref{contract.binom}).

Let $\lambda x^\beta + \mu x^\gamma$ be an arbitrary binomial in $I(L,\rho) \cap \Abq$, with $\lambda,\mu \in \KK$ and $\beta,\gamma \in \Gamma^+$. Since $I(L,\rho)$ contains no monomials, $\lambda$ and $\mu$ are both nonzero. Now $\lambda x^\beta + \mu x^\gamma = \lambda (x^\beta - \nu x^\gamma)$ where $\nu := - \lambda^{-1} \mu \in \kx$, and it suffices to show that $x^\beta - \nu x^\gamma \in I_A(L,\rho)$. Set
$$\delta := \bigl( \min(\beta_1,\gamma_1) , \dots , \min(\beta_n,\gamma_n) \bigr), \qquad \beta' := \beta - \delta, \qquad \gamma' := \gamma - \delta,$$
so that $\beta = \beta' + \delta$ and $\gamma = \gamma' + \delta$ with $\beta', \gamma', \delta \in \Gamma^+$ and $\beta' \perp \gamma'$. Now
$$x^\beta - \nu x^\gamma = d(\beta',\delta)^{-1} x^{\beta'} x^\delta - \nu d(\gamma', \delta)^{-1} x^{\gamma'} x^\delta = d(\beta', \delta)^{-1} \bigl( x^{\beta'} - \nu' x^{\gamma'} \bigr) x^\delta,$$
with $\nu' := \nu d(\beta',\delta) d(\gamma', \delta)^{-1} \in \kx$ and $x^{\beta'} - \nu' x^{\gamma'} \in I(L,\rho) \cap \Abq$, and it is enough to show that $x^{\beta'} - \nu' x^{\gamma'} \in I_A(L,\rho)$. Thus, with no loss of generality, $\beta \perp \gamma$.

Set $\alpha := \beta - \gamma$, so that $\alpha_+ = \beta$ and $\alpha_- = \gamma$. Now
$$( x^{\beta} - \nu x^{\gamma} ) (x^\gamma)^{-1} = d(\gamma, \gamma) x^\beta x^{- \gamma} - \nu = d(\gamma, \gamma) d(\beta, - \gamma) x^\alpha - \nu = d(\alpha, \gamma)^{-1} x^\alpha - \nu,$$
and so $x^\alpha - \nu d(\alpha,\gamma) \in I(L,\rho)$. Theorem \ref{biject.pairs.binom.Tq} implies that $\alpha \in L$ and $\nu d(\alpha, \gamma) = c(\alpha)^{-1} \rho(\alpha)$. Therefore
$$x^{\beta} - \nu x^{\gamma} = x^{\alpha_+} - c(\alpha)^{-1} d(\alpha, \alpha_-)^{-1} \rho(\alpha) x^{\alpha_-} \in I_A(L,\rho),$$
as desired.
\end{proof}
In view of Lemma \ref{ILrho.basis}, one might expect that if $(\beta(1), \dots, \beta(r))$ is a basis for $L$, then $I_A(L,\rho)$ can be generated by the elements
$$x^{\beta(l)_+} - c(\beta(l))^{-1} d(\beta(l), \beta(l)_-)^{-1} \rho(\beta(l)) x^{\beta(l)_-}, \qquad l \in [1,r].$$
However, this fails in general, as the following example, adapted from \cite[Example 2.16]{Mil}, shows.

\begin{example}  \label{eg.basisnongen}
Let $n=4$ and
$$\bfq = \begin{bmatrix}
1&q&q^2&q^3\\  q^{-1}&1&q&q^2\\  q^{-2}&q^{-1}&1&q\\  q^{-3}&q^{-2}&q^{-1}&1  \end{bmatrix}$$
where $q \in \kx$ is a non-root of unity. One checks that $\beta(1) := (1,-2,1,0)$ and $\beta(2) := (0,1,-2,1)$ form a basis for $\Gamma_Z$, so that $\Zq$ is a Laurent polynomial ring in two variables, $x^{\beta(1)} = x_1 x_2^{-2} x_3$ and $x^{\beta(2)} = x_2 x_3^{-2} x_4$. Set $L := \Gamma_Z$, and take $\rho \in \Hom(L,\kx)$ to be the homomorphism such that $\rho(\beta(l)) = q^{-2}$ for $l=1,2$.

Clearly $c(\beta(l)) = 1$ for $l=1,2$, and one checks that $d(\beta(l), \beta(l)_-) = q^{-2}$. Thus, $I_A(L,\rho)$ contains the elements
\begin{equation}  \label{expl.plusminus.basis}
\begin{aligned}
x^{\beta(l)_+} - c(\beta(l))^{-1} &d(\beta(l), \beta(l)_-)^{-1} \rho(\beta(l)) x^{\beta(l)_-}  \\
&= x^{\beta(l)_+} - x^{\beta(l)_-} = \begin{cases}  x_1x_3 - x_2^2 &(l=1)\\ x_2x_4 - x_3^2 &(l=2).
\end{cases}
\end{aligned}
\end{equation}
We show that these two elements do not generate $I_A(L,\rho)$. 

Consider the element $\beta := (1,-1,-1,1) = \beta(1) + \beta(2)$ in $L$. One checks that
$$c(\beta) = q^{-1}, \qquad d(\beta, \beta_-) = q^{-2}, \qquad \rho(\beta) = q^{-4},$$
and so $I_A(L,\rho)$ contains the element
$$x^{\beta_+} - c(\beta)^{-1} d(\beta, \beta_-)^{-1} \rho(\beta) x^{\beta_-} = x_1x_4 - q^{-1} x_2x_3 \,.$$
Since $x_1 x_3 - x_2^2$ and $x_2 x_4 - x_3^2$ lie in the ideal $\langle x_2, x_3\rangle$ of $\Abq$ but $x_1 x_4 - q^{-1} x_2x_3$ does not,
we conclude that the elements listed in \eqref{expl.plusminus.basis} do not generate $I_A(L,\rho)$.
\end{example}

\subsection{}  \label{xbull}
Let $x_\bullet$ denote the multiplicatively closed subset of $\Abq$ generated by $\{ x_1,\dots,x_n\}$, and set
\begin{equation}  \label{def.specprim0}
\begin{aligned}
\Spec_0 \Abq &:= \{ P \in \Spec \Abq \mid P \cap x_\bullet = \varnothing \} = \{ P \in \Spec \Abq \mid  x_1,\dots,x_n \notin P \}  \\
\Prim_0 \Abq &:= \Prim \Abq \cap \Spec_0 \Abq \,.
\end{aligned}
\end{equation}
(The second equality in the first line of \eqref{def.specprim0} relies on the fact that the $x_i$ are normal elements of $\Abq$.) In case $\KK$ is infinite, the sets $\Spec_0 \Abq$ and $\Prim_0 \Abq$ are the \emph{$0$-strata} of $\Spec \Abq$ and $\Prim \Abq$ in the sense of \cite[Definitions II.2.1]{BG}, relative to a natural action of the torus $\HH := (\kx)^n$ on $\Abq$.

Now $\Tbq$ is the left and right Ore localization $\Abq[x_\bullet^{-1}]$. Since $\Abq$ is noetherian, extension (localization) and contraction provide inverse bijections
\begin{equation}  \label{primes.localization}
\Spec_0 \Abq \longleftrightarrow \Spec \Tbq\,, \qquad P \longmapsto P \Tbq\,, \qquad Q \cap \Abq \longmapsfrom Q
\end{equation}
(e.g., \cite[Theorem 10.20]{GW}).  In particular, any radical ideal of $\Tbq$ contracts to a radical ideal of $\Abq$.

\begin{corollary}  \label{IALrho.radical}
Let $L$ be a sublattice of $\Gamma_Z$ and $\rho \in \Hom(L,\kx)$. If $\charr \KK = 0$, then $I_A(L,\rho)$ is a radical ideal of $\Abq$.
\end{corollary}

\begin{proof}
Proposition \ref{binomTq.radical} and Theorem \ref{IcapAq}.
\end{proof}

This corollary does not imply, however, that in characteristic zero binomial ideals containing no monomials are necessarily radical ideals.  For instance, take the binomial ideal $\langle x_1^2x_2 + x_1^2 \rangle$ in the commutative polynomial ring $\KK[x_1,x_2]$ (in any characteristic).

\subsection{}
For any ideal $I$ of $\Abq$, let $(I:x_\bullet)$ denote the \emph{$x_\bullet$-closure} of $I$, that is, the ideal
\begin{equation}  \label{def.I:xbull}
(I:x_\bullet) := \{ a \in \Abq \mid ay \in I\ \text{for some}\ y \in x_\bullet \} = I\Tbq \cap \Abq \,.
\end{equation}

\begin{corollary}  \label{xbull.closure}
If $I$ is a binomial ideal of $\Abq$ not containing any monomials, there exist a unique sublattice $L$ of $\Gamma_Z$ and a unique $\rho \in \Hom(L,\kx)$ such that $(I:x_\bullet) = I_A(L,\rho)$. If, further, $I$ is a prime ideal, then $I = I_A(L,\rho)$.
\end{corollary}

\begin{proof}
By assumption, $I$ is disjoint from $x_\bullet$, whence $I\Tbq$ is a proper binomial ideal of $\Tbq$. By Theorem \ref{biject.pairs.binom.Tq}, there exist a unique sublattice $L$ of $\Gamma_Z$ and a unique $\rho \in \Hom(L,\kx)$ such that $I\Tbq = I(L,\rho)$. Thus $(I:x_\bullet) = I\Tbq \cap \Abq = I_A(L,\rho)$ by Theorem \ref{IcapAq}. In case $I$ is a prime ideal, $I = (I: x_\bullet)$.
\end{proof}

We now harvest information about the ideals $I_A(L,\rho)$ from the results of Section \ref{binomTq}, via localization and stratification.

\begin{theorem}  \label{prim0Aq}
{\rm(a)} For any $\rho \in \Hom(\Gamma_Z,\kx)$, the ideal $I_A(\Gamma_Z,\rho)$ is a primitive ideal of $\Abq$.

{\rm(b)} If $\KK$ is algebraically closed, then every primitive ideal of $\Abq$ that contains no monomials has the form $I_A(\Gamma_Z,\rho)$ for some $\rho \in \Hom(\Gamma_Z,\kx)$.
\end{theorem}

\begin{proof}
(a) By Theorems \ref{maxTq} and \ref{IcapAq}, $I(\Gamma_Z,\rho)$ is a maximal ideal of $\Tbq$ and $I_A(\Gamma_Z,\rho)$ localizes to $I(\Gamma_Z,\rho)$. Consequently, $I_A(\Gamma_Z,\rho)$ is a maximal element of $\Spec_0 \Abq$, and so any prime ideal of $\Abq$ that strictly contains $I_A(\Gamma_Z,\rho)$ must contain some $x_i$. Thus
$$\bigcap \{ P \in \Spec \Abq \mid P \supsetneqq I_A(\Gamma_Z,\rho) \} \supseteq I_A(\Gamma_Z,\rho) + \langle x_1 \cdots x_n \rangle \supsetneqq I_A(\Gamma_Z,\rho),$$
which shows that $I_A(\Gamma_Z,\rho)$ is locally closed in $\Spec \Abq$. By \cite[Lemma II.7.15 and Corollary II.7.18]{BG}, $I_A(\Gamma_Z,\rho) \in \Prim \Abq$.

(b) Given the assumption $\KK = \KKbar$, the set $\Prim_0 \Abq$ consists precisely of the maximal elements of $\Spec_0 \Abq$ \cite[Corollary II.8.5]{BG}. Thus, if $P$ is a primitive ideal of $\Abq$ that contains no monomials, then $P$ is a maximal element of $\Spec_0 \Abq$, whence $P$ localizes to a maximal ideal $P\Tbq$ of $\Tbq$. By Theorem \ref{maxTq}, $P \Tbq = I(\Gamma_Z,\rho)$ for some $\rho \in \Hom(\Gamma_Z,\kx)$, and therefore $P = P \Tbq \cap \Abq = I_A(\Gamma_Z,\rho)$ due to Theorem \ref{IcapAq}.
\end{proof}

\begin{theorem}  \label{prime.IALrho}
Let $L$ be a sublattice of $\Gamma_Z$ and  $\rho \in \Hom(L,\kx)$.

{\rm(a)} If $\Gamma_Z/L$ is torsionfree, then $I_A(L,\rho)$ is a prime ideal of $\Abq$. If, further, $\Gamma/L$ is torsionfree, then $I_A(L,\rho)$ is completely prime.

{\rm(b)} If $\KK$ is algebraically closed, then $I_A(L,\rho)$ is a prime ideal of $\Abq$ if and only if $\Gamma_Z/L$ is torsionfree. 

{\rm(c)} If $\KK$ is algebraically closed and $\Gamma/\Gamma_Z$ is torsionfree, then all prime binomial ideals of $\Abq$ that contain no monomials are completely prime.
\end{theorem}

\begin{proof}
(a)(b) Theorems \ref{prime.binomTq} and \ref{IcapAq}.

(c) If $P$ is a prime binomial ideal of $\Abq$ that contains no monomials, then $P \Tbq$ is a prime binomial ideal of $\Tbq$ that contracts to $P$. By Theorems \ref{binom.Tq} and \ref{IcapAq}, $P \Tbq = I(L,\rho)$ for some sublattice $L$ of $\Gamma_Z$ and some $\rho \in \Hom(L,\kx)$ and so $P = I_A(L,\rho)$. Part (b) (or just Theorem \ref{prime.binomTq}(b)) implies that $\Gamma_Z/L$ is torsionfree. Since $\Gamma/\Gamma_Z$ is torsionfree, so is $\Gamma/L$, and therefore $P$ is completely prime by part (a).
\end{proof}

\begin{theorem}  \label{rad.binomAq}
Let $L$ be a sublattice of $\Gamma_Z$ and $\rho \in \Hom(L,\kx)$.

{\rm(a)} If $\charr\KK = 0$, then $\sqrt{I_A(L,\rho)} = I_A(L,\rho)$.

{\rm(b)} Assume that $\KK$ is perfect of characteristic $p > 0$, and let $L'_p/L$ be the $p$-torsion subgroup of $\Gamma_Z/L$. There is a unique map $\rho'$ in $\Hom(L'_p, \kx)$ extending $\rho$, and $\sqrt{I_A(L,\rho)} = I_A(L'_p, \rho')$.

{\rm(c)} Assume that $\KK$ is algebraically closed. Let $L'/L$ be the torsion subgroup of $\Gamma_Z/L$, and let $\rho'_1,\dots,\rho'_m \in \Hom(L',\kx)$ be the distinct extensions of $\rho$ to $L'$. The prime ideals of $\Abq$ minimal over $I_A(L,\rho)$ are the $I_A(L', \rho'_j)$ for $j \in [ 1,m ]$. All of these prime ideals have the same height, equal to $\rank L$.
\end{theorem}

\begin{proof} (a)(b) Set $J := \sqrt{I(L,\rho)}$, which equals $I(L,\rho)$ in case (a) and $I(L'_p, \rho')$ in case (b), by Theorem \ref{rad.binom}. Observe that $J \cap \Abq$ is a radical ideal of $\Abq$ containing $I_A(L,\rho)$. Now $J^k \subseteq I(L,\rho)$ for some $k \in \Znn$, whence $(J \cap \Abq)^k \subseteq I_A(L,\rho)$, and consequently $J \cap \Abq = \sqrt{I_A(L,\rho)}$. Theorem \ref{IcapAq} therefore implies that $\sqrt{I_A(L,\rho)}$ equals $I_A(L,\rho)$ in case (a) and $I_A(L'_p, \rho')$ in case (b).

(c) By Theorem \ref{minprimes.binom}, the prime ideals of $\Tbq$ minimal over $I(L,\rho)$ are the ideals $Q_j := I(L', \rho'_j)$ for $j \in [ 1,m ]$. It follows from \eqref{primes.localization} and Theorem \ref{IcapAq} that the ideals $P_j := Q_j \cap \Abq = I_A(L', \rho'_j)$ for $j \in [ 1,m ]$ are prime ideals of $\Abq$ minimal over $I_A(L,\rho)$. As in the proof of (a),
$$\sqrt{I_A(L,\rho)} = \sqrt{I(L,\rho)} \cap \Abq = Q_1 \cap \cdots \cap Q_m \cap \Abq = P_1 \cap \cdots \cap P_m \,.$$
Consequently, any prime ideal $P$ of $\Abq$ minimal over $I_A(L,\rho)$ must contain some $P_j$, whence $P = P_j$ by minimality. Finally,
$$\hgt I_A(L', \rho'_j) = \hgt I(L', \rho'_j) = \rank L \qquad \forall\; j \in [1,m]$$
by \eqref{primes.localization}, Corollary \ref{hts.Tq}, and Theorem \ref{minprimes.binom}.
\end{proof}

\sectionnew{General binomial ideals in $\Abq$}  \label{genbinom.Aq}
General binomial ideals of $\Abq$ -- ones that may contain monomials -- can be approached by way of special quotients of $\Abq$, since $\Abq$ modulo the ideal generated by a subset of the generators $x_i$ is again a quantum affine space. This allows results from Section \ref{binom.Aq.w/o.monom} to be applied to prime binomial ideals and, provided $\KK$ is algebraically closed, to primitive ideals and to primes minimal over arbitrary binomial ideals. Finally, we prove that if $\KK$ is perfect, then the radical of any binomial ideal of $\Abq$ is binomial.

\subsection{}
For any subset $J \subseteq [ 1,n ]$, define the following:
\begin{align*}
\Tbq[J] &:= \KK \langle x_j^{\pm1} \mid j \in J \rangle \subseteq \Tbq  &\Abq[J] &:= \KK \langle x_j \mid j \in J \rangle  \subseteq \Abq \\
\Gamma[J] &:= \{ \alpha \in \Gamma \mid \alpha_i = 0\ \forall\, i \notin J \}  &\Gamma[J]^+ &:= \Gamma[J] \cap \Gamma^+  \\
\Gamma[J]_Z &:= \{ \alpha \in \Gamma[J] \mid x^\alpha \in Z(\Tbq[J]) \} &I_A[\neg J] &:= \langle x_j \mid j \in [1,n] \setminus J \rangle \vartriangleleft \Abq \,. 
\end{align*}
Then $\Abq[J]$ is a quantum affine space $\calO_{\bfq[J]}(\KK^{|J|})$ and $\Tbq[J]$ is the corresponding quantum torus $\calO_{\bfq[J]}((\kx)^{|J|})$, where $\bfq[J]$ is the $J \times J$ submatrix of $\bfq$. Moreover,
$$Z(\Tbq[J]) = \bigoplus_{\alpha \in \Gamma[J]_Z} \KK x^\alpha\,.$$
There is a $\KK$-algebra retraction 
$$\pi_J : \Abq \longrightarrow \Abq[J] \quad \text{such that} \quad \pi_J(x_j) = 
\begin{cases} x_j &(j \in J) \\  0 &(j \notin J) \end{cases}$$
with $\ker \pi_J = I_A[\neg J]$, which induces a canonical isomorphism of $\Abq/I_A[\neg J]$ onto $\Abq[J]$.

\subsection{}
Given $J \subseteq[1,n]$, choose and fix a basis $(\beta_{J,1}, \dots, \beta_{J,r(J)})$ for $\Gamma[J]_Z$ (allowing $r(J) = 0$ in case $\Gamma[J]_Z = 0$), write $\KK\Gamma[J]_Z$ as a Laurent polynomial ring of the form
$$\KK \Gamma[J]_Z = \KK [ z_{J,1}^{\pm1}, \dots, z_{J,r(J)}^{\pm1} ],$$
and set
$$z_{J,\gamma} := z_{J,1}^{m_1} \cdots z_{J,r(J)}^{m_{r(J)}} \qquad \forall\; \gamma = m_1 \beta_{J,1} + \cdots + m_{r(J)} \beta_{J,r(J)} \in \Gamma[J]_Z \,,$$
so that $(z_{J,\gamma} \mid \gamma \in \Gamma[J]_Z )$ is a $\KK$-basis for $\KK \Gamma[J]_Z$ and $z_{J,\gamma} z_{J,\delta} = z_{J,\gamma+\delta}$ for $\gamma, \delta \in \Gamma[J]_Z$. There is a $\KK$-algebra isomorphism
$$\phi_J : \KK \Gamma[J]_Z \overset{\cong}{\longrightarrow} Z(\Tbq[J]) \quad \text{such that} \quad \phi_J(z_{J,i}) = x^{\beta_{J,i}}\ \text{for}\ i \in [ 1,r(J) ].$$
There are scalars $c_J(\gamma) \in \kx$ such that
$$\phi_J(z_{J,\gamma}) = c_J(\gamma) x^\gamma \qquad \forall\, \gamma \in \Gamma[J]_Z \,,$$
and, as in \eqref{d.ab}, it follows from \eqref{d.prod} that
$$d(\alpha, \beta) = c_J(\alpha+\beta) c_J(\alpha)^{-1} c_J(\beta)^{-1} \qquad \forall\, \alpha,\beta \in \Gamma[J]_Z \,.$$

\subsection{}
For any sublattice $L$ of $\Gamma[J]_Z$ and $\rho \in \Hom(L,\kx)$, define
$$I_{J,A}(L,\rho) := I_A[\neg J] + \langle x^{\alpha_+} - c_J(\alpha)^{-1} d(\alpha, \alpha_-)^{-1} \rho(\alpha) x^{\alpha_-} \mid \alpha \in L \rangle \vartriangleleft \Abq \,.$$
We may also express $I_{J,A}(L,\rho)$ as $I_A[\neg J] + I'$ where $I'$ is the ideal of $\Abq[J]$ generated by
$$\{ x^{\alpha_+} - c_J(\alpha)^{-1} d(\alpha, \alpha_-)^{-1} \rho(\alpha) x^{\alpha_-} \mid \alpha \in L \}.$$

\begin{theorem}  \label{primJ.Aq}
{\rm(a)} For any $J \subseteq [ 1,n ]$ and $\rho \in \Hom( \Gamma[J]_Z, \kx)$, the ideal $I_{J,A}(\Gamma[J]_Z,\rho)$ is a primitive ideal of $\Abq$.

{\rm(b)} Assume that $\KK$ is algebraically closed. Then every primitive ideal of $\Abq$ has the form $I_{J,A}(\Gamma[J]_Z,\rho)$
for some $J \subseteq [ 1,n ]$ and $\rho \in \Hom( \Gamma[J]_Z, \kx)$.

In particular, all primitive ideals of $\Abq$ are binomial ideals in this case. 
\end{theorem}

\begin{proof} 
The primitive ideals of $\Abq$ are the ideals $I_A[\neg J] + Q$ for $J \subseteq [ 1,n ]$ and $Q \in \Prim_0 \Abq[J]$. The theorem thus follows from Theorem \ref{prim0Aq}.
\end{proof} 

\begin{theorem}  \label{prime.IJA}
Let $L$ be a sublattice of $\Gamma[J]_Z$ for some $J\subseteq [ 1,n ]$ and $\rho \in \Hom(L,\kx)$.

{\rm(a)} If $\Gamma[J]_Z/L$ is torsionfree, then $I_{J,A}(L,\rho)$ is a prime ideal of $\Abq$. If, further, $\Gamma[J]/L$ is torsionfree, then $I_{J,A}(L,\rho)$ is completely prime.

{\rm(b)} If $\KK$ is algebraically closed, then $I_{J,A}(L,\rho)$ is a prime ideal of $\Abq$ if and only if $\Gamma[J]_Z/L$ is torsionfree.

{\rm(c)} If $\KK$ is algebraically closed and $\Gamma[J]/\Gamma[J]_Z$ is torsionfree for all $J \subseteq [ 1,n ]$, then all prime binomial ideals of $\Abq$ are completely prime.
\end{theorem} 

\begin{proof}
The prime ideals of $\Abq$ are the ideals $I_A[\neg J] + Q$ for $J \subseteq [ 1,n ]$ and $Q \in \Spec_0 \Abq[J]$. Apply
Theorem \ref{prime.IALrho}.
\end{proof}

\begin{corollary}  \label{prime.binom.Aq}
Assume that $\KK$ is algebraically closed, let $P$ be a binomial ideal in $\Abq$, and set $J := \{ j \in [ 1,n ] \mid x_j \notin P \}$. Then $P$ is prime if and only if $P = I_{J,A}(L,\rho)$ for some sublattice $L$ of $\Gamma[J]_Z$ such that $\Gamma[J]_Z/L$ is torsionfree and some $\rho \in \Hom(L,\kx)$.
\end{corollary}

\begin{proof} $(\Longleftarrow)$: Theorem \ref{prime.IJA}(a).

$(\Longrightarrow)$: We can write $P = I_A[\neg J] + Q$ where $Q := \pi_J(P) \vartriangleleft \Abq[J]$. Observe that $Q$ is a binomial ideal and a prime ideal of $\Abq[J]$. Since the $x_i$ are normal elements of $\Abq$, we see that $x^\alpha \notin P$ for all $\alpha \in \Gamma[J]$, and so $x^\alpha \notin Q$ for such $\alpha$.

Corollary \ref{xbull.closure} now implies that $P = I_{J,A}(L,\rho)$ for some sublattice $L$ of $\Gamma[J]_Z$ and some $\rho \in \Hom(L,\kx)$, after which Theorem \ref{prime.IJA}(b) implies that $\Gamma[J]_Z/L$ is torsionfree.
\end{proof}

For $J' \subseteq [1,n]$, set
$$
\Spec_{J'} \Abq := \bigl\{ P \in \Spec \Abq \bigm| P \cap \{x_1,\dots,x_n\} = \{x_j \mid j \in J' \} \bigr\}.
$$
These sets partition $\Spec \Abq$.  The elements $x_j$ for $j \in J := [1,n] \setminus J'$ are normal in $\Abq$ and hence are regular modulo any $P \in \Spec_{J'} \Abq$.  It follows that all primes in $\Spec_{J'} \Abq$ are disjoint from the set
$$
x_{J\bullet} := \text{the multiplicative subset of}\; \Abq\; \text{generated by}\; \{ x_j \mid j \in J \},
$$
and that the elements of $x_{J\bullet}$ are regular modulo any such prime.

\begin{lemma}  \label{piJP.min}
Let $I$ be an ideal of $\Abq$, let $J \subseteq [1,n]$ and $J' := [1,n] \setminus J$, and let $P \in \Spec_{J'} \Abq$.

{\rm(a)} $\pi_J(P)$ is a prime ideal of $\Abq[J]$ disjoint from $x_{J\bullet}$.

{\rm(b)}  If $\pi_J(P)$ is minimal over $(\pi_J(I) : x_{J\bullet})$, then $\pi_J(P)$ is also minimal over $\pi_J(I)$.  Moreover, $P$ is minimal over $I + I_A[\neg J]$ and over $(I + I_A[\neg J] : x_{J\bullet})$.

{\rm(c)}  If $P$ is minimal over $I$, then $\pi_J(P)$ is minimal over $\pi_J(I)$ and minimal over $(\pi_J(I) : x_{J\bullet})$.
\end{lemma}

\begin{proof} 
(a) By assumption, $I_A[\neg J] = \langle x_j \mid j \in J' \rangle \subseteq P$, whence $\pi_J(P) \in \Spec \Abq[J]$ and $\pi_J^{-1}(\pi_J(P)) = P$.  As noted above, $P$ is disjoint from $x_{J\bullet}$ and the elements of $x_{J\bullet}$ are regular modulo $P$.  It follows that $\pi_J(P)$ is disjoint from $x_{J\bullet}$ and the elements of $x_{J\bullet}$ are regular modulo $\pi_J(P)$.

(b) First consider a prime $Q$ of $\Abq[J]$ such that $\pi_J(I) \subseteq Q \subseteq \pi_J(P)$.  The elements of $x_{J\bullet}$ are normal in $\Abq[J]$ and lie outside $\pi_J(P)$ by (a), so they must be regular modulo $Q$.  Since $\pi_J(I) \subseteq Q$, it follows that $(\pi_J(I) : x_{J\bullet}) \subseteq Q$, and hence $Q = \pi_J(P)$ by the minimality of $\pi_J(P)$.  Thus $\pi_J(P)$ is minimal over $\pi_J(I)$.

The minimality statements concerning $P$ follow because $\pi_J^{-1}(\pi_J(I)) = I + I_A[\neg J]$ and $\pi_J^{-1} (\pi_J(I) : x_{J\bullet}) = (I + I_A[\neg J] : x_{J\bullet})$.

(c) Suppose $Q \in \Spec \Abq[J]$ with $\pi_J(I) \subseteq Q \subseteq \pi_J(P)$.  Then $\pi_J^{-1}(Q)$ is a prime of $\Abq$ lying between $I$ and $P$, and hence $\pi_J^{-1}(Q) = P$.  It follows that $Q = \pi_J(P)$, proving that $\pi_J(P)$ is minimal over $\pi_J(I)$.

For the final statement, we just need to show that $\pi_J(P)$ contains $(\pi_J(I) : x_{J\bullet})$.  If $a \in (\pi_J(I) : x_{J\bullet})$, then $ya \in \pi_J(I)$ for some $y \in x_{J\bullet}$.  Then $ya \in \pi_J(P)$, whence $a \in \pi_J(P)$ because $y$ is regular modulo $\pi_J(P)$.  This establishes the required inclusion.
\end{proof}

\begin{theorem}  \label{min.over.binom}
Assume that $\KK$ is algebraically closed.
If $I$ is a binomial ideal of $\Abq$, then every prime ideal of $\Abq$ minimal over $I$ is binomial.
\end{theorem}

\begin{proof} 
Let $P$ be a prime of $\Abq$ minimal over $I$, and let $J'$ be the unique subset of $[1,n]$ such that $P \in \Spec_{J'} \Abq$.  By Lemma \ref{piJP.min}(a)(c), $\pi_J(P)$ is minimal over $(\pi_J(I) : x_{J\bullet})$, where $J := [1,n] \setminus J'$, and $\pi_J(P)$ is disjoint from $x_{J\bullet}$.  Since $\pi_J(I)$ is disjoint from $x_{J\bullet}$, Corollary \ref{xbull.closure} implies that $(\pi_J(I) : x_{J\bullet}) = I_{\Abq[J]}(L,\rho)$ for some sublattice $L$ of $\Gamma[J]_Z$ and $\rho \in \Hom(L,\kx)$.  Theorem \ref{rad.binomAq}(c) then shows that $\pi_J(P)$ is a binomial ideal of $\Abq[J]$.  Since $P = \pi_J^{-1}(\pi_J(P))$ and $\ker \pi_J = I_A[\neg J]$ is a binomial ideal of $\Abq$, we conclude that $P$ is a binomial ideal of $\Abq$.
\end{proof}

Our remaining goal is to prove that radicals of binomial ideals in $\Abq$ are binomial when $\KK$ is perfect. We use noncommutative versions of several lemmas from \cite{EiSt}.

\begin{lemma}  \label{ES1.3}
If $I$ is a binomial ideal of $\Abq$, then $I \cap \Abq[J]$ is a binomial ideal of $\Abq[J]$ for any $J \subseteq [ 1,n ]$.
\end{lemma}

\begin{proof} Fix $J \subseteq [ 1,n ]$. 

Grade $\Abq/I$ by $S := \Gamma^+/{\sim}$ as in Lemma \ref{Aq.mod.binom}. Then $\Abq[J]/(I \cap \Abq[J])$ is isomorphic to the subalgebra $R$ of $\Abq/I$ generated by the cosets $y_j := x_j + I$ for $j \in J$. Since the $y_j$ are homogeneous with respect to the $S$-grading of $\Abq/I$, the algebra $R$ is also graded by $S$, with homogeneous components $R_s = R \cap (\Abq / I)_s$. 

Now $y_j y_k = q_{jk} y_k y_j$ for all $j,k \in J$. The natural map $\psi : \Abq[J] \rightarrow R$ sends $x_j$ to $y_j$ for $j \in J$, and $\ker\psi = I \cap \Abq[J]$. Lemma \ref{kerphi.binom} therefore shows that $I \cap \Abq[J]$ is a binomial ideal of $\Abq[J]$.
\end{proof}

\begin{lemma}  \label{ES1.56}
Let $I$ and $I'$ be binomial ideals and $J_1,\dots,J_m$ monomial ideals of $\Abq$. 

{\rm(a)} $(I+J_1) \cap \cdots \cap (I+J_m) = I+J$ for some monomial ideal $J$ of $\Abq$.

{\rm(b)} $(I+I') \cap (I+J_1) \cap \cdots \cap (I+J_m)$ is a binomial ideal
of $\Abq$.
\end{lemma}

\begin{proof}
(a) Set $R := \Abq/I$. With notation as in Lemma \ref{Aq.mod.binom}, $R$ is graded by the monoid $S := \Gamma^+/{\sim}$, with components $R_s = \KK y^\alpha$ for $s \in S$ and $\alpha \in s$. Set $S_\bullet := \{ s\in S \mid R_s \ne 0 \}$ and fix $\alpha(s) \in s$ for $s\in S_\bullet$. Then $\calB := (y^{\alpha(s)} \mid s \in S_\bullet)$ is a $\KK$-basis for $R$.

For $i \in [ 1,m ]$, the ideal $\Jbar_i := (I+J_i)/I$ of $R$ is generated by some collection of $y^\alpha$s, and hence also spanned by such a collection. Thus, $\Jbar_i$ has a $\KK$-basis $\calB_i := (y^{\alpha(t)} \mid t \in T_i)$ for some $T_i \subseteq S_\bullet$. Since the $\calB_i$ are all part of the same basis $\calB$ for $R$, the intersection $\Jbar := \bigcap_{i=1}^m \Jbar_i$ has a $\KK$-basis $( y^{\alpha(t)} \mid t \in T )$ where $T := \bigcap_{i=1}^m T_i$. Now $\Jbar = \bigl( \bigcap_{i=1}^m (I+J_i) \bigr) /I$, and so $\bigcap_{i=1}^m (I+J_i) = I+J^0$ where $J^0$ is the $\KK$-span of $\{ x^{\alpha(t)} \mid t \in T \}$. We conclude that $\bigcap_{i=1}^m (I+J_i) = I+ \langle x^{\alpha(t)} \mid t \in T \rangle$.

(b) In view of part (a), only the case $m=1$ needs to be proved.
We proceed as in \cite[Corollary 1.5]{EiSt}, working with a polynomial ring $\Abq[t]$, which is a quantum affine space of rank $n+1$. Let $I^+ := I[t] + t I'[t] + (1-t) J_1[t]$ be the ideal of $\Abq[t]$ generated by $I + tI' + (1-t)J_1$, and observe that $I^+$ is a binomial ideal. We claim that $I^+ \cap \Abq = (I+I') \cap (I+J_1)$. On one hand, if $a,b \in I$, $c \in I'$, and $d \in J_1$ with $a+c = b+d$, then
$$
a+c = a + tc + (1-t)(b+d-a) = [a + (1-t)(b-a)] + tc + (1-t)d \in I^+ \cap \Abq\,.
$$
If $\ell \in I^+ \cap \Abq$, then $\ell = f + tg + (1-t)h$ for some $f \in I[t]$, $g \in I'[t]$, and $h \in J_1[t]$. Then $\ell = \ell(0) = f(0) + h(0) \in I+J_1$ and $\ell = \ell(1) = f(1) + g(1) \in I + I'$, whence $\ell \in (I+I') \cap (I+J_1)$. This establishes the claim, and therefore Lemma \ref{ES1.3} implies that $ (I+I') \cap (I+J_1)$ is a binomial ideal.
\end{proof}

\begin{lemma}  \label{ES3.2}
Let $R$ be an arbitrary ring, $I$ an ideal of $R$, and $x_1,\dots,x_n$ some normal elements of $R$.  Define $x_\bullet$ and $(I:x_\bullet)$ as in {\rm\S\ref{xbull}} and \eqref{def.I:xbull}.  Then
\begin{equation}  \label{sqrtI}
\sqrt{I} = \sqrt{(I:x_\bullet)} \cap \sqrt{I+\langle x_1\rangle} \cap \cdots \cap \sqrt{I+\langle x_n \rangle}.
\end{equation}
\end{lemma}

\begin{proof} This is proved exactly as \cite[Lemma 3.2]{EiSt}, given that any normal element $z \in R$ that lies outside a prime ideal $P$ of $R$ is regular modulo $P$.
\end{proof}

\begin{lemma}  \label{ES3.3}
Let $I$ be a binomial ideal of $\Abq$. Set $\calA_i := \Abq[[1,n] \setminus \{i\}]$ and $I_i := I \cap \calA_i$ for $i \in [ 1,n ]$. Then $I + \langle x_i \rangle = \Abq I_i \Abq + \langle x_i \rangle + M_i \Abq$ for some monomial ideal $M_i$ of $\calA_i$.
\end{lemma}

\begin{remark*}
Since $M_i$ is generated by monomials, which are normal elements in $\Abq$, the right ideal $M_i \Abq$ is actually a two-sided ideal. This may not hold for $I_i \Abq$, so we take $\Abq I_i \Abq$ to be sure to get the two-sided ideal generated by $I_i$. On the other hand, the retraction $\pi_{[1,n] \setminus \{i\}}$ maps $\Abq I_i \Abq + \langle x_i \rangle + M_i \Abq$ onto $I_i+M_i$, so we can also express this ideal in the form $I_i + M_i + \langle x_i \rangle$.
\end{remark*}

\begin{proof}
This is proved exactly as \cite[Lemma 3.3]{EiSt}, modulo obvious changes of notation.
\end{proof}

\subsection{}  \label{radAq.to.radAi}
Observe that for any $i \in [1,n]$, the algebra $\Abq$ may be presented as a skew polynomial ring $\calA_i[x_i; \sigma_i]$ where $\calA_i := \Abq[ [1,n] \setminus \{i\}]$ and $\sigma_i$ is an automorphism of $\calA_i$. Consequently, all prime ideals of $\Abq$ contract to finite intersections of prime ideals of $\calA_i$ (e.g., \cite[Remark 5$^*$, Lemmas 1.2, 1.3]{GolMic}). Thus, all radical ideals of $\Abq$ contract to radical ideals of $\calA_i$.

\begin{proposition}  \label{ES3.4}
Let $I$ be a radical binomial ideal and $M$ a monomial ideal of $\Abq$. Then $\sqrt{I+M} = I+J$ for some monomial ideal $J$ of $\Abq$.
\end{proposition}

\begin{proof}
We follow the line of \cite[Proposition 3.4]{EiSt}, by induction on $n$. The case $n=0$ is trivial. Now assume that $n>0$.

Since there is nothing to prove if $M=0$, we may assume that there exists a monomial in $M$. Then $(I+M : x_\bullet) = \Abq$ and Lemma \ref{ES3.2} implies that
$$\sqrt{I+M} = \bigcap_{i=1}^n \sqrt{ I+M+\langle x_i \rangle }.$$
By Lemma \ref{ES1.56}(a), it suffices to show that each $\sqrt{ I+M+\langle x_i \rangle } = I+J_i$ for some monomial ideal $J_i$.

Fix $i \in [ 1,n ]$ and set $\calA_i := \Abq[ [1,n] \setminus \{i\}]$ and $\pi_i := \pi_{[1,n] \setminus \{i\}} : \Abq \rightarrow \calA_i$. By \S\ref{radAq.to.radAi}, $I_i := I \cap \calA_i$ is a radical ideal of $\calA_i$.

Note that $M + \langle x_i \rangle = M'_i \Abq + \langle x_i \rangle$ where $M'_i$ is the ideal of $\calA_i$ generated by those monomials in $M$ which do not involve $x_i$.  By Lemma \ref{ES3.3}, $I + \langle x_i \rangle = \Abq I_i \Abq + M''_i \Abq + \langle x_i \rangle$ for some monomial ideal $M''_i$ of $\calA_i$, and thus
$$I+M+\langle x_i \rangle = \Abq I_i \Abq + M_i \Abq + \langle x_i \rangle, \qquad M_i := M'_i + M''_i \,.$$
By induction, $\sqrt{ I_i+M_i } = I_i + N_i$ for some monomial ideal $N_i$ of $\calA_i$, where $\sqrt{ I_i+M_i }$ denotes the radical of $I_i+M_i$ in $\calA_i$. Now
\begin{align*}
\sqrt{ I+M+\langle x_i \rangle } &= \sqrt{ \Abq I_i \Abq + M_i \Abq + \langle x_i \rangle } = \sqrt{I_i+M_i} + \langle x_i \rangle  \\
&= I_i + N_i + \langle x_i \rangle \subseteq I + N_i \Abq + \langle x_i \rangle \subseteq \sqrt{ I+M+\langle x_i \rangle } \,,
\end{align*}
using that $\pi_i( \Abq I_i \Abq + M_i \Abq + \langle x_i \rangle ) = I_i+M_i$ and $N_i \subseteq \sqrt{I_i+M_i} \subseteq \sqrt{ I+M+\langle x_i \rangle }$. Thus $\sqrt{ I+M+\langle x_i \rangle } = I+J_i$ where $J_i := N_i \Abq + \langle x_i \rangle$ is a monomial ideal of $\Abq$, as desired.
\end{proof}

\begin{theorem}  \label{ES3.1}
 Assume that $\KK$ is perfect. If $I$ is any binomial ideal of $\Abq$, then $\sqrt{I}$ is also binomial.
 \end{theorem}
 
 \begin{proof}
 We again proceed by induction on $n$, the case $n=0$ being trivial. Assume now that $n>0$. We may replace $I$ by any binomial ideal lying between $I$ and $\sqrt I$, so without loss of generality, $I$ equals the sum of all binomial ideals contained in $\sqrt I$.
 
\textbf{Claim 1}: The ideal $I' := \sqrt{(I:x_\bullet)}$ is binomial.
 
 If there exists a monomial in $I$, then $(I:x_\bullet) = \Abq$ and $I' = \Abq$. If $I$ contains no monomials, then by Corollary \ref{xbull.closure}, $(I:x_\bullet)$ is a proper binomial ideal $I_A(L,\rho)$, and Theorem \ref{rad.binomAq} implies that $I'$ is binomial. 
 
 As in the previous proof, set $\calA_i := \Abq[ [1,n] \setminus \{i\}]$ and $\pi_i := \pi_{[1,n] \setminus \{i\}} : \Abq \rightarrow \calA_i$ for $i \in [ 1,n ]$, and let $I_i := I \cap \calA_i$. 
 
\textbf{Claim 2}: $I_i$ is a radical ideal of $\calA_i$.

Lemma \ref{ES1.3} implies that $I_i$ is a binomial ideal of $\calA_i$, and so by induction $\sqrt{I_i}$ is binomial. By \S\ref{radAq.to.radAi}, $\sqrt I \cap \calA_i$ is a radical ideal of $\calA_i$, so it contains $\sqrt{I_i}$. Consequently, the binomial ideal $\Abq \sqrt{I_i} \Abq$ is contained in $\sqrt I$, and thus $\Abq \sqrt{I_i} \Abq \subseteq I$ by our assumptions on $I$. It follows that $\sqrt{I_i} \subseteq I_i$, yielding the claim.

\textbf{Claim 3}: $\sqrt{I+ \langle x_i \rangle} = I+J_i$ for some monomial ideal $J_i$ of $\Abq$.

First note that $\Abq I_i \Abq + \langle x_i \rangle = \pi_i^{-1}(I_i)$, which is a radical ideal of $\Abq$ due to Claim 2. According to Lemma \ref{ES3.3},
$$I + \langle x_i \rangle = \Abq I_i \Abq + \langle x_i \rangle + M_i \Abq$$
for some monomial ideal $M_i$ of $\calA_i$. Applying Proposition \ref{ES3.4} to the radical binomial ideal $\Abq I_i \Abq + \langle x_i \rangle$ and the monomial ideal $M_i \Abq$, we obtain
$$\sqrt{ I + \langle x_i \rangle } = \sqrt{ \Abq I_i \Abq + \langle x_i \rangle + M_i \Abq } = \Abq I_i \Abq + \langle x_i \rangle + J'_i$$
for some monomial ideal $J'_i$. Since $\Abq I_i \Abq + \langle x_i \rangle \subseteq I + \langle x_i \rangle \subseteq \sqrt{ I + \langle x_i \rangle }$, it follows that $\sqrt{I+ \langle x_i \rangle} = I+J_i$ where $J_i := \langle x_i \rangle + J'_i$.
 
 Combining Claims 1 and 3 with Lemma \ref{ES3.2}, and writing $I' = I+I'$, we obtain
 $$\sqrt{I} = (I+I') \cap (I+J_1) \cap \cdots \cap (I+J_m).$$
 Therefore Lemma \ref{ES1.56}(b) implies that $\sqrt I$ is binomial.
 \end{proof}

\begin{remark}
It is known that Theorem \ref{ES3.1} does not hold when $\KK$ is not perfect, even in the commutative case. By \cite[Lemma 29]{BGN}, the polynomial ring $\KK[x_1,x_2]$ has a binomial ideal whose radical is not binomial (cf.~Example \ref{rad.nonbinom.Tq}).
\end{remark}

\section*{Acknowledgement}
We thank the referee for extensive comments and suggestions, and for raising questions about binomial algebras.



\begin{thebibliography}{AF}

\bibitem{BGN} E. Becker, R. Grobe and M. Niermann, \emph{Radicals of binomial ideals}, J. Pure Appl. Algebra \textbf{117}\&\textbf{118} (1997), 41--79.

\bibitem{BG} K.A. Brown and K.R. Goodearl, \emph{Lectures on Algebraic Quantum Groups}, Advanced Courses in Mathematics CRM Barcelona,
Basel (2002) Birkh\"auser.

\bibitem{Coh} P.M. Cohn, \emph{Algebra}, Vol. 2, London (1977) Wiley.
  
\bibitem{EiSt} D. Eisenbud and B. Sturmfels, \emph{Binomial ideals}, Duke Math. J. \textbf{84} (1996), 1--45.

\bibitem{GolMic} A.W. Goldie and and G. Michler, \emph{Ore extensions and polycyclic group rings}, J. London Math. Soc. (2) \textbf{9} (1974), 337--345.

\bibitem{GLet94} K.R. Goodearl and E.S. Letzter, \emph{Prime factor algebras of the coordinate ring of quantum matrices}, Proc. Amer. Math. Soc. \textbf{121} (1994), 1017--1025.
  
  \bibitem{GW} K.R. Goodearl and R.B. Warfield, Jr., \emph{An Introduction to Noncommutative Noetherian Rings}, 2nd. Ed., London Math. Soc. Student Text Series \textbf{61}, Cambridge (2004) Cambridge Univ. Press.
  
\bibitem{Ing} C. Ingalls, \emph{Quantum toric varieties}, manuscript (2003).
  
 \bibitem{Kar} G. Karpilovsky, \emph{The Algebraic Structure of Crossed Products}, Math. Studies 142, Amsterdam (1987) North-Holland.

\bibitem{Mil} E. Miller, \emph{Theory and applications of lattice point methods for binomial ideals}, in \emph{Combinatorial Aspects of Commutative Algebra 99 and Algebraic Geometry}, G. Fl\o ystad et al. (eds.), Abel Symposia \textbf{6} (2011), 99--154.

\bibitem{MR} J.C. McConnell and J.C. Robson, \emph{Noncommutative Noetherian Rings}, Grad. Studies in Math. 30, Providence, R.I. (2001) Amer. Math. Soc.

\bibitem{MoPa} S. Montgomery and D.S. Passman, \emph{Crossed products over prime rings}, Israel J. Math. \textbf{31} (1978), 224--256.

\bibitem{NVO} C. N\v ast\v asescu and F. Van Oystaeyen, \emph{Methods of Graded Ring Theory}, Berlin (2004) Springer-Verlag.

\bibitem{Pas} D.S. Passman, \emph{The Algebraic Structure of Group Rings}, New York (1977) Wiley.

\bibitem{Pas2} \bysame, \emph{Semiprime and prime crossed products}, J. Algebra \textbf{83} (1983), 158--178.

\bibitem{RZ} L. Rigal and P. Zadunaisky, \emph{Twisted semigroup algebras}, Algebras and Rep. Theory \textbf{18} (2015), 1155--1186.
  
\end{thebibliography}
\end{document}